\documentclass[11pt, reqno]{amsart}

\usepackage[top=3.75cm, bottom=3cm, left=3cm, right=3cm]{geometry}
\usepackage{amsmath}
\usepackage{amsthm}
\usepackage{amsfonts}
\usepackage{amssymb}
\usepackage{mathtools}
\usepackage{bm}
\usepackage[colorlinks=true, citecolor=citation, urlcolor=citation, linkcolor=reference]{hyperref}
\usepackage{color}
\usepackage[all,cmtip]{xy}
\usepackage{enumitem}
\usepackage{tikz}  
\usepackage{graphicx}
\usepackage{scalerel}
\usepackage{comment}
\usepackage[mathscr]{eucal}
\usepackage{stmaryrd}

\definecolor{citation}{rgb}{0,.40,.80}
\definecolor{reference}{rgb}{.80,0,.40}

\numberwithin{equation}{section}

\theoremstyle{plain}
\newtheorem{theorem}{Theorem}[section]
\newtheorem{lemma}[theorem]{Lemma}
\newtheorem{proposition}[theorem]{Proposition}
\newtheorem{corollary}[theorem]{Corollary}
\newtheorem{conjecture}[theorem]{Conjecture}

\theoremstyle{definition}
\newtheorem{definition}[theorem]{Definition}
\newtheorem{example}[theorem]{Example}
\newtheorem{remark}[theorem]{Remark}

\newcommand{\st}{\mid} 
\newcommand{\set}[1]{\left\{ \, #1 \, \right\}}

\newcommand{\Db}{\mathrm{D^b}}

\newcommand{\Gal}{\mathrm{Gal}}

\DeclareMathOperator{\Pic}{Pic}

\DeclareMathOperator{\Spec}{Spec}

\newcommand{\toptriv}{\mathrm{tt}}

\newcommand{\Tate}{\mathrm{Tate}}

\newcommand{\id}{\mathrm{id}}
\newcommand{\pr}{\mathrm{pr}}

\newcommand{\ind}{\mathrm{ind}}

\DeclareMathOperator{\characteristic}{char}

\DeclareMathOperator{\Br}{Br}
\DeclareMathOperator{\Brtop}{Br^{\mathrm{top}}}

\DeclareMathOperator{\CH}{CH}

\newcommand{\rtop}{\mathrm{top}}

\newcommand{\an}{\mathrm{an}}

\newcommand{\Ktop}[1][]{\rK_{#1}^{\rtop}}

\newcommand{\per}{\mathrm{per}}

\newcommand{\tors}{\mathrm{tors}}

\newcommand{\cl}{\mathrm{cl}}

\newcommand{\et}{\mathrm{\acute{e}t}}

\newcommand{\Az}{\mathrm{Az}}

\newcommand{\Hdg}{\mathrm{Hdg}}

\newcommand{\PGL}{\mathrm{PGL}}

\DeclareMathOperator{\codim}{codim}

\DeclareMathOperator{\rk}{{rk}}

\newcommand{\cO}{\mathcal{O}}
\newcommand{\cA}{\mathcal{A}}

\newcommand{\cX}{\mathcal{X}}

\newcommand{\rH}{\mathrm{H}}

\newcommand{\rK}{\mathrm{K}}

\newcommand{\rM}{\mathrm{M}}

\newcommand{\rR}{\mathrm{R}}

\newcommand{\rZ}{\mathrm{Z}}

\newcommand{\bC}{\mathbf{C}}

\newcommand{\bF}{\mathbf{F}}
\newcommand{\bG}{\mathbf{G}}

\newcommand{\bZ}{\mathbf{Z}}
\newcommand{\bP}{\mathbf{P}}
\newcommand{\bQ}{\mathbf{Q}}

\newcommand{\bmu}{\bm{\mu}}

\begin{document}

\title{The period-index problem and Hodge theory}

\author{Aise Johan de Jong}
\address{Department of Mathematics, Columbia University, New York, NY 10027 \smallskip}
\email{dejong@math.columbia.edu}

\author{Alexander Perry}
\address{Department of Mathematics, University of Michigan, Ann Arbor, MI 48109 \smallskip}
\email{arper@umich.edu}

\thanks{
A.\ J.\ de Jong was partially supported by NSF grant DMS-2052934.  
A.\ Perry was partially supported by NSF grants DMS-2112747, DMS-2052750, and DMS-2143271, and a Sloan Research Fellowship.} 


\begin{abstract} 
Conditional on the Lefschetz standard conjecture in degree 2, 
we prove that the index of a Brauer class on a smooth projective 
variety divides a fixed power of its period, uniformly in smooth families. 
In the other direction, we reinterpret in more classical terms recent work of Hotchkiss 
which gives Hodge-theoretic lower bounds on the index of Brauer classes. 
We also prove versions of our results over arbitrary algebraically closed base fields, 
and as an application construct qualitatively new counterexamples to the integral Tate conjecture. 
\end{abstract}

\maketitle


\section{Introduction} 

The goal of this paper is to bring Hodge theory to bear on the period-index problem. 

\subsection{The period-index problem} 
Let $K$ be a field. 
Recall that a central simple algebra over $K$ is a unital associative $K$-algebra which has finite dimension over $K$, no nontrivial two-sided ideals, and center $K$. 
Two such algebras $A$ and $B$ are called Morita equivalent if there exists an isomorphism of matrix algebras $\rM_r(A) \cong \rM_s(B)$ for some positive integers $r$ and $s$. 
The Brauer group $\Br(K)$ is the set of all central simple algebras over $K$ modulo Morita equivalence, with group operation given by tensor product. 
By a classical result, any central simple algebra over $K$ is Morita equivalent to a unique central division algebra over $K$, so $\Br(K)$ can be thought of as a group classifying such division algebras. 

There are two important numerical invariants measuring the complexity of a 
Brauer class $\alpha \in \Br(K)$: 
its \emph{period} $\per(\alpha)$, equal to to order of $\alpha$ in $\Br(K)$; and its 
\emph{index} $\ind(\alpha)$, equal to the integer $\sqrt{\dim_K(D)}$ where $D$ is the central division algebra of class $\alpha$ 
(whose dimension, like any central simple algebra, is necessarily a square). 
As is well-known, $\per(\alpha)$ and $\ind(\alpha)$ share the same prime factors, and $\per(\alpha)$ divides $\ind(\alpha)$. 
The \emph{period-index problem} is to determine an integer $e$ so that $\ind(\alpha)$ divides $\per(\alpha)^e$. 
Work of many authors over the past century has suggested the following precise conjecture, which appears to have first been raised in print by Colliot-Th\'{e}l\`{e}ne \cite{CT-PI} (see also \cite{CT-bourbaki} and \cite{lieblich-period-index}). 

\begin{conjecture}[Period-index conjecture]
\label{conjecture-pi} 
Let $K$ be a field of finite transcendence degree $d$ over an algebraically closed field $k$. 
For all $\alpha \in \Br(K)$, 
$\ind(\alpha)$ divides $\per(\alpha)^{d-1}$. 
\end{conjecture} 

The conjecture is vacuously true 
for $d \leq 1$ since then $\Br(K) = 0$ by an elementary argument when $d = 0$ and by Tsen's theorem when $d = 1$. 
It is also true for $d = 2$ by \cite{dJ-period-index} when $\per(\alpha)$ is prime to the characteristic of $k$ and by \cite{lieblich-period-index, dJ-starr} in general.  
In higher dimensions, the conjecture is wide open: 
it is not even known for a single field $K$ of transcendence degree $d \geq 3$ 
that there exists an integer $e$ such $\ind(\alpha)$ divides $\per(\alpha)^e$ for all $\alpha \in \Br(K)$. 
In this paper, we study the higher-dimensional problem in the global setting.  

Namely, let $X$ be a smooth projective variety over $k$. 
By Grothendieck \cite{grothendieck-brauer}, there is an extension $\Br(X)$ of the Brauer group 
to this situation, classifying Azumaya algebras on $X$ up to Morita equivalence. 
If $K$ is a field of finite transcendence degree over $k$ and $X$ is a model for $K$, then the restriction map $\Br(X) \to \Br(K)$ is an isomorphism onto the subgroup of unramified Brauer classes (and in particular the image is independent of the model $X$). 
By discriminant avoidance \cite{dJ-starr}, Conjecture~\ref{conjecture-pi} 
reduces to the unramified case, in that if it holds for all unramified Brauer classes on all $K$ then it holds in general. 
In this sense, unramified classes are the crucial ones from the perspective of the period-index problem. 
This focuses attention on the following statement, where for $\alpha \in \Br(X)$ 
we denote by $\per(\alpha)$ and $\ind(\alpha)$ the period and index of the restriction of $\alpha$ to the generic point. 

\begin{conjecture}
\label{conjecture-pi-bound}
Let $X$ be a smooth projective variety over an algebraically closed field $k$. 
There exists a positive integer $e$ such that for all $\alpha \in \Br(X)$, 
$\ind(\alpha)$ divides $\per(\alpha)^e$. 
\end{conjecture} 

Even this weakened version of Conjecture~\ref{conjecture-pi} is wide open when $\dim(X) \geq 3$. 

\subsection{Upper bounds on the index via Hodge theory}
Our first main result, stated below as Theorem~\ref{theorem-pi-bound}, gives a positive answer to 
Conjecture~\ref{conjecture-pi-bound} conditional on one of Grothendieck's standard conjectures \cite{grothendieck-standard}. 
In the introduction we focus on the case of complex varieties for simplicity, but in the body of the paper we also prove results over general algebraically closed fields. 

The setup for the relevant motivic conjecture is as follows. 
Let $X$ be a smooth projective complex variety of dimension~$n$, with $h \in \rH^2(X, \bQ)$ the first Chern class of an ample line bundle. 
For $0 \leq i \leq n$, the hard Lefschetz theorem shows that the map 
\begin{equation*}
    L^{n-i} \coloneqq (-) \cup h^{n-i} \colon \rH^i(X, \bQ) \to \rH^{2n-i}(X, \bQ) 
\end{equation*}
is an isomorphism. 
The Lefschetz standard conjecture predicts that the inverse isomorphism is algebraic: 

\begin{conjecture}[Lefschetz standard conjecture in degree $i$]
\label{conjecture-lefschetz}
For $X$, $h$, and $i$ as above, there exists a codimension $i$ cycle $\lambda \in \CH^i(X \times X)_{\bQ}$ such that the induced map 
\begin{equation*}
    \lambda^* \colon \rH^{2n-i}(X, \bQ) \to \rH^{i}(X, \bQ) 
\end{equation*}
is the inverse of $L^{n-i}$.  
\end{conjecture}

\begin{remark}
\label{remark-lc-independent-h}
Conjecture~\ref{conjecture-lefschetz} is in fact independent of the choice of polarization $h$. 
Indeed, it is equivalent to the existence of $\lambda \in \CH^i(X \times X)_{\bQ}$ such that $\lambda^* \colon \rH^{2n-i}(X, \bQ) \to \rH^{i}(X, \bQ)$ is an isomorphism; see  \cite[Theorem 4.1]{kleiman-standard} or \cite[Lemma~6]{charles-remarks}. 
\end{remark}

\begin{remark}
\label{remark-lc-results}
Conjecture~\ref{conjecture-lefschetz} is wide open in general, but there are positive results for some special classes of varieties. 
For instance, the Lefschetz standard conjecture is known in all degrees when $X$ is any of the following: 
a variety of dimension at most $2$ \cite{kleiman-weil}, 
an abelian variety \cite{lieberman}, 
a threefold of Kodaira dimension less than $3$
\cite{tankeev-I, tankeev-II}, a Hilbert scheme of points on a smooth projective 
surface \cite{arapura}, or a holomorphic symplectic variety of K3$^{[n]}$-type (i.e. deformation equivalent to a Hilbert scheme of points on a K3 surface) \cite{charles-markman}. There are also partial results on the Lefschetz standard conjecture in degree $2$ for holomorphic symplectic varieties that admit a covering by Lagrangians  \cite{voisin-lagrangian}.  
\end{remark}

Now we can state our result. 

\begin{theorem}
\label{theorem-pi-bound}
Let $X \to S$ be a smooth proper morphism of complex varieties. 
Assume that the very general fiber of $X \to S$ is projective and satisfies the Lefschetz standard conjecture in degree $2$.  
Then there exists a positive integer~$e$ such that for all geometric points $\bar{s}$ of  $S$ and $\alpha \in \Br(X_{\bar{s}})$, 
$\ind(\alpha)$ divides $\per(\alpha)^e$. 
In particular, in the absolute case when $S = \Spec(\bC)$, if the Lefschetz standard conjecture holds in degree $2$ for $X$ then so does Conjecture~\ref{conjecture-pi-bound}. 
\end{theorem} 

As usual, ``the very general fiber'' means $X_s$ for $s \in S(\bC)$ lying in the complement of some countable union of proper Zariski closed subsets of $S$. 

We also prove an analogous result over an arbitrary algebraically closed base field $k$ (with the role of the very general fiber replaced by the geometric generic fiber); see Theorem \ref{theorem-pi-bound-absolute-k} for the absolute statement  and Theorem~\ref{theorem-pi-bound-relative-k} for the relative one. 
One subtlety is that we need to assume the Lefschetz standard conjecture in degree $2$ for $\ell$-adic cohomology, uniformly in all primes $\ell \neq \characteristic(k)$; see Conjecture~\ref{conjecture-l-lefschetz} and Remark~\ref{remark-l-lefschetz}. 

\begin{remark}
The integer $e$ given by Theorem~\ref{theorem-pi-bound} only depends on the relative dimension of $X \to S$, the size of the torsion in $\rH^2(X_0, \bZ)$ and $\rH^3(X_0, \bZ)$ for a fiber over a point $0 \in S(\bC)$, 
and some simple numerical data associated to the cycle $\lambda \in \CH^i(X_s \times X_s)_{\bQ}$ on a very general fiber $X_s$ given by the Lefschetz standard conjecture,  
cf. Proposition~\ref{proposition-cycle}.
\end{remark}

For families of surfaces Theorem~\ref{theorem-pi-bound} is not interesting, as the period-index conjecture is already known in dimension $2$, but it is interesting in higher dimensions. 
As noted in Remark~\ref{remark-lc-results}, there are various cases where the Lefschetz standard conjecture in degree $2$ is known, and therefore where Theorem~\ref{theorem-pi-bound} becomes unconditional. 
For an abelian variety $X$, one can show directly that $\ind(\alpha)$ divides $\per(\alpha)^{\dim(X)}$ for all $\alpha \in \Br(X)$, so in this case we get nothing new. 
However, the other cases mentioned in Remark~\ref{remark-lc-results} lead to the following instances of Conjecture~\ref{conjecture-pi-bound}, which as far as we know are new. 

\begin{corollary}
Let $X \to S$ be a smooth projective morphism of complex varieties, whose fibers are one of the following: threefolds of Kodaira dimension less than $3$, Hilbert schemes of points on smooth projective surfaces, or holomorphic symplectic varieties of K3$^{[n]}$-type. 
Then there exists a positive integer $e$ such that for all geometric points $\bar{s}$ of $S$ and $\alpha \in \Br(X_{\bar{s}})$, $\ind(\alpha)$ divides $\per(\alpha)^e$. 
\end{corollary}

Let us briefly explain the idea behind Theorem~\ref{theorem-pi-bound}, focusing on the absolute case where $S = \Spec(\bC)$ to which the general case can be reduced. 
Let $i \colon Y \hookrightarrow X$ be a complete intersection surface of class $h^{n-2}$, where $n = \dim(X)$ and $h$ is an ample class on $X$. 
The algebraicity of 
the inverse Lefschetz map 
$(L^{n-2})^{-1} \colon \rH^{2n-2}(X, \bQ) \to \rH^2(X, \bQ)$ leads to an integral cycle $\Gamma \in \CH^2(X \times Y)$ such that $\Gamma^* \circ i^*$ acts by multiplication by an integer $m$ on $\Br(X)$. Using the period-index conjecture for surfaces, we show that this implies there are constants $C$ and $N$ such that $\ind(m\alpha)$ divides $C \cdot \per(\alpha)^N$ for all $\alpha \in \Br(X)$ (Lemma~\ref{lemma-surface}). 
This allows us to focus on classes of bounded period, for which a result of Matzri \cite{matzri} shows that the index divides a power of the period depending only on the dimension of the variety. 

\subsection{Lower bounds on the index via Hodge theory} 
In the second part of the paper, we turn to a question in the opposite direction of the above results: how to produce lower bounds on the index of Brauer classes? 
In the unramified setting, the first results of this kind were obtained in  \cite{CT-examples, kresch}, with examples showing that Conjecture~\ref{conjecture-pi} is sharp. 

Recently, Hotchkiss \cite{hotchkiss-pi} introduced a beautiful new perspective on this question. 
Let $X$ be a smooth projective complex variety and $\alpha \in \Br(X)$. 
By the Hodge theory of categories developed in \cite{perry-CY2}, the twisted derived category $\Db(X, \alpha)$ has an associated integral Hodge structure $\Ktop[0](X, \alpha)$, which receives a map $\rK_0(X, \alpha) \to \Ktop[0](X, \alpha)$ from the Grothendieck group of $\Db(X, \alpha)$ that factors through the subgroup of integral Hodge classes. 
Motivated by the fact that $\ind(\alpha)$ equals the minimal positive rank of an element of $\rK_0(X, \alpha)$, 
Hotchkiss defines the \emph{Hodge-theoretic index}
$\ind_{\Hdg}(\alpha)$ of $\alpha$ as the minimal positive rank of a Hodge class in $\Ktop[0](X, \alpha)$. 
By definition 
$\ind_{\Hdg}(\alpha)$ 
divides $\ind(\alpha)$, so lower bounds on the former give lower bounds on the latter. This method is especially effective when $\alpha$ is \emph{topologically trivial}, or equivalently can be written as $\alpha = \exp(B)$ for a \emph{rational $B$-field} $B \in \rH^2(X, \bQ)$ (see Remark~\ref{remark-rational-B}), 
as in this case Hotchkiss calculates $\Ktop[0](X, \alpha)$ explicitly in terms of the Hodge theory of $X$. 

Using the Hodge theory of Severi--Brauer varieties, we explain a different point of view on  
Hotchkiss's restrictions on the index of a topologically trivial Brauer class. 
For $B \in \rH^2(X, \bQ)$ and a positive integer $e$, 
define for $1 \leq i \leq \min \set{ e, \dim(X) }$ the polynomial 
\begin{equation}
\label{pBe}
    p_i^{B,e}(x_1, \dots, x_i) =
    \binom{e}{i}B^i + 
    \sum_{j=1}^{i} \binom{e-j}{i-j} B^{i-j} x_j  \in \rH^{2*}(X, \bQ)[x_1, \dots, x_i]. 
\end{equation}
Note that $p_i^{B,e}$ is weighted homogeneous of degree $2i$ if $x_j$ is given weight $2j$. 
Below we use the terminology that a class in $\rH^*(X, \bQ)$ is \emph{integral} if it lies in the image of $\rH^*(X, \bZ)$. 

\begin{theorem}
\label{theorem-lower-bound} 
Let $X$ be a smooth projective complex variety. 
Let $\alpha \in \Br(X)$ be a topologically trivial class, and let $B \in \rH^{2}(X, \bQ)$ be a rational $B$-field for $\alpha$. 
If $e$ is a positive integer such that $\ind(\alpha)$ divides $e$, then 
there exist Hodge classes 
\begin{equation*}
c_j \in \rH^{j,j}(X, \bQ), \quad 1 \leq j \leq \min \set{ e, \dim(X) }, 
\end{equation*} 
such that the class 
$p_i^{B,e}(c_1, \dots, c_i) \in \rH^{2i}(X, \bQ)$ is integral for all $1 \leq i \leq \min \set{ e, \dim(X) }$. 
\end{theorem}

The existence of Hodge classes $c_j \in \rH^{j,j}(X, \bQ)$ as in the conclusion of the theorem gives an obstruction to $\ind(\alpha)$ dividing $e$, which can be nontrivial in general. 
Indeed, 
for $\per(\alpha) = e = 2$   Theorem~\ref{theorem-lower-bound} recovers Kresch's obstruction to $\ind(\alpha) = 2$ \cite{kresch} (see Example~\ref{example-kresch}), which is known to be effective for some examples. 
Using the obstruction for $e = \per(\alpha)^{\dim(X)-2}$, we also construct in Proposition~\ref{proposition-index-av} examples with $X$ an abelian variety (even a product of elliptic curves) where the period-index conjecture is sharp, 
i.e. $\ind(\alpha) = \per(\alpha)^{\dim(X) -1}$. 

The case $e = \per(\alpha)^{\dim(X)-1}$ of Theorem~\ref{theorem-lower-bound} is particularly intriguing, as it gives a potential obstruction to the period-index conjecture. 
When $\per(\alpha)$ is sufficiently large relative to $\dim(X)$ this obstruction vanishes (Lemma~\ref{lemma-PI-obstruction-vanish}), but for $\per(\alpha)$ small relative to $\dim(X)$ it leads to mysterious identities among integral Hodge classes. 
These identities either hold universally or the period-index conjecture is false, but it is 
unclear what to expect; see \S\ref{section-obstruction-PI} for further discussion. 

Our main motivation for writing down Theorem~\ref{theorem-lower-bound} is that it gives an elementary perspective, in particular not involving the Hodge theory of categories, on the method of \cite{hotchkiss-pi}. 
However, we regard Hotchkiss's perspective via $\ind_{\Hdg}(\alpha)$ as more fundamental for two reasons. 
First, we prove in Lemma~\ref{lemma-obstruction-comparison} that if $\ind_{\Hdg}(\alpha)$ divides $e$, then Hodge classes as in the conclusion of Theorem~\ref{theorem-lower-bound} exist; thus, $\ind_{\Hdg}(\alpha)$ is at least as powerful an obstruction. 
Second, $\ind_{\Hdg}(\alpha)$ leads to a strategy for \emph{proving} the period-index conjecture in cases when it divides $\per(\alpha)^{\dim(X)-1}$, by using sheaf-theoretic methods to show the algebraicity of a Hodge class in $\Ktop[0](X, \alpha)$ of the correct rank; for abelian threefolds, this strategy is successfully carried out in \cite{hotchkiss-perry}. 

One advantage of our approach is that it extends directly to other base fields\footnote{It should also be possible to extend the approach of \cite{hotchkiss-pi} to other base fields, but this requires some more machinery.}, 
as we explain in \S\ref{section-lower-bound-k}. 
As an application, we construct counterexamples to the integral Tate conjecture which are topologically as simple as possible. 

\begin{theorem}
\label{theorem-ITC}
Let $p$ and $\ell$ be distinct primes. 
There exists 
a smooth projective variety $P$ of dimension $\ell^\ell + \ell$ over $\overline{\bF}_p$ with 
$\rH^*_{\et}(P, \bZ_{\ell})$ torsion-free on which the $\ell$-adic integral Tate conjecture in codimension $\ell^\ell-1$ fails. 
More precisely, there exists such a $P$ which is a Severi--Brauer variety of relative dimension $\ell^\ell-1$ over a product of $\ell+1$ pairwise non-isogenous elliptic curves. 
\end{theorem} 

See Theorem~\ref{theorem-ITC-precise} for a more precise version of Theorem~\ref{theorem-ITC}, 
and Conjecture~\ref{conjecture-ITC} for the formulation of the $\ell$-adic integral Tate conjecture. 
For $\ell > 5$, Theorem~\ref{theorem-ITC} gives the first counterexamples to this conjecture over $\overline{\bF}_p$ with a torsion-free cohomology ring; for $\ell \in \set{2, 3, 5}$ such counterexamples were only very recently constructed by Benoist \cite{benoist} via a completely different method. 
We note that the base of the Severi--Brauer variety $P$ above does satisfy the $\ell$-adic integral Tate conjecture (Lemma~\ref{lemma-generic-product-Fp}). 

Our proof of Theorem~\ref{theorem-ITC} is inspired by a counterexample to the integral Hodge conjecture on a Severi--Brauer variety from~\cite{hotchkiss-pi}. 
The idea is to construct a Brauer class $\alpha$ on $X$ with $\ind(\alpha)$ large, and then construct a Tate class on a Severi--Brauer variety $P \to X$ for $\alpha$ whose algebraicity would force $\ind(\alpha)$ to be small. 
Notably, for the first step, in Proposition~\ref{proposition-index-av-Fp} 
we use the $\ell$-adic version of Theorem~\ref{theorem-lower-bound} to construct Brauer classes $\alpha$ on products $X$ of pairwise non-isogenous elliptic curves over $\overline{\bF}_p$ with $\ind(\alpha) = \per(\alpha)^{\dim(X)-1}$. 

Finally, we note that in Theorem~\ref{theorem-IHC} we also prove an analog of Theorem~\ref{theorem-ITC} over the complex numbers for the integral Hodge conjecture. 

\subsection{Organization of the paper} 
In \S\ref{section-preliminaries} we gather some preliminary results about Brauer groups that are used throughout the paper. 
In \S\ref{section-upper-bound} we prove Theorem~\ref{theorem-pi-bound}, and in \S\ref{section-ub-other-fields} we prove an analogous result over  general base fields. 
In \S\ref{section-lower-bound} we prove Theorem~\ref{theorem-lower-bound} and derive the applications mentioned above. 
In \S\ref{section-lower-bound-k} we prove an analog of Theorem~\ref{theorem-lower-bound} over general base fields and discuss some of its applications, including the proof of Theorem~\ref{theorem-ITC}. 

\subsection{Conventions} 
\label{conventions} 
For an integral scheme $X$, we write $k(X)$ for its function field.
For a dominant morphism $f \colon X \to Y$ of integral schemes, we write $f_{k(Y)} \colon \Spec(k(X)) \to \Spec(k(Y))$ for the induced morphism on spectra of function fields. 

A variety $X$ over a field $k$ is an integral scheme which is separated and of finite type over~$k$. 
We write $\CH^i(X)$ for the Chow group of $X$ in degree $i$, and $\CH^i(X)_{\bQ} = \CH^i(X) \otimes \bQ$ for its rationalization. 

If $X$ is a complex variety, we write $\rH^*(X, A)$ for the singular cohomology with coefficients in $A$ of the analytification $X^{\an}$. 
If $X$ is a smooth proper complex variety, we write $\rH^{i,i}(X, \bZ)$ and $\rH^{i,i}(X, \bQ)$ for the integral and rational Hodge classes in degree $2i$. 
We say that a class in $\rH^{2i}(X, \bZ)$ is \emph{algebraic} if it lies in the image of the cycle class map $\CH^{i}(X) \to \rH^{2i}(X, \bZ)$. 

If $X$ is a smooth proper variety over the algebraic closure $k$ of a finitely generated field and $\ell \neq \characteristic(k)$ is a prime, 
we write $\rH^{2i}_{\et}(X, \bZ_{\ell}(i))^{\Tate}$ and $\rH^{2i}_{\et}(X, \bQ_{\ell}(i))^{\Tate}$ for the $\ell$-adic integral and rational Tate classes in degree $2i$. 
We say that a class in $\rH_{\et}^{2i}(X, \bZ_{\ell}(i))$ is \emph{algebraic} if it lies in the image of the cycle class map $\CH^{i}(X) \otimes \bZ_{\ell} \to \rH_{\et}^{2i}(X, \bZ_{\ell}(i))$. 

\subsection{Acknowledgements} 
We heartily thank James Hotchkiss for inspiring conversations related to this work. 
We also benefited from helpful discussions with Will Sawin and Stefan Schreieder. 


\section{Preliminaries} 
\label{section-preliminaries} 

The purpose of this section is to fix notation and gather preliminaries for the rest of the paper. 

\subsection{Brauer groups} 
We recall some basic results about Brauer groups, referring the reader to one of the many sources on this material, like  \cite{grothendieck-brauer,  CT-brauer}, for more details. 

Let $X$ be a scheme. 
The \emph{Brauer--Azumaya group} $\Br_{\Az}(X)$ is defined as the group of Azumaya algebras on $X$ modulo Morita equivalence. 
This is a torsion group if $X$ is quasi-compact. 
Moreover, there is a natural injection $\Br_{\Az}(X) \hookrightarrow \rH^2_{\et}(X, \bG_m)$, 
which in the quasi-compact case factors through the torsion subgroup $\Br(X) \coloneqq \rH^2_{\et}(X, \bG_m)_{\tors}$ known as the \emph{cohomological Brauer group}\footnote{Sometimes in the literature this group is instead denoted $\Br'(X)$, while $\Br_{\Az}(X)$ is denoted $\Br(X)$.}. 
When $X$ admits an ample line bundle, the map $\Br_{\Az}(X) \to \Br(X)$ is an isomorphism by a theorem of Gabber \cite{dj-gabber}. 
If $X$ is integral, regular, and noetherian,  restriction to the generic point gives an injection $\rH^2_{\et}(X, \bG_m) \hookrightarrow \rH^2_{\et}(k(X), \bG_m)$; in particular, $\Br(X) = \rH^2_{\et}(X, \bG_m)$ in this case. 
Summarizing, if $X$ is a smooth quasi-projective variety then the three possible notions of the Brauer group, 
$\Br_{\Az}(X)$, $\Br(X)$, and $\rH^2_{\et}(X, \bG_m)$, all coincide and inject into $\Br(k(X))$. 

Next we discuss functoriality. 
Let $f \colon X \to Y$ be a morphism of schemes. 
Then there is a pullback map $f^* \colon \Br(Y) \to \Br(X)$ (sometimes also called restriction) induced by pullback on \'{e}tale cohomology. 
For $\alpha \in \Br(Y)$ we sometimes write $\alpha\vert_X$, or simply $\alpha \vert_A$ when $X = \Spec(A)$ is affine, for $f^*(\alpha)$. 

When $f \colon X \to Y$ is finite locally free of constant rank, there is also a pushforward morphism $f_* \colon \Br(X) \to \Br(Y)$ (sometimes also called corestriction), which is compatible with composition and base change and such that $f_*f^*$ is multiplication by $\deg(f)$ on $\Br(X)$. 
Let us record a  description of the opposite composition, $f^*f_*$, for field extensions, which follows from standard Galois cohomology (see e.g. \cite{GS}). 

\begin{lemma}
\label{lemma-push-pull} 
Let $M/L/K$ be field extensions of finite degree such that $M/K$ is Galois, and let 
$\Spec(M) \xrightarrow{g} \Spec(L) \xrightarrow{f} \Spec(K)$ be the corresponding maps. 
Let 
$\sigma_1, \dots, \sigma_d \in \Gal(M/K)$ be a choice of representatives for the left cosets of  $\Gal(M/L)$ in $\Gal(M/K)$. 
Then for $\alpha \in \Br(L)$ we have 
\begin{equation*}
    g^*f^*f_*(\alpha) = \sum_{i} \sigma_i^*g^*(\alpha). 
\end{equation*}
\end{lemma}

In \S\ref{section-correspondences-Br} below we will discuss functoriality with respect to appropriate correspondences, obtained by combining pullback and pushforward, after recalling the corresponding story for cohomology in \S\ref{section-correspondences}. 

\subsection{Action of correspondences on cohomology} 
\label{section-correspondences} 
Let $X$ and $Y$ be smooth 
varieties over an algebraically closed field $k$, with $Y$ proper. 
Then correspondences between $X$ and $Y$ act on cohomology.
For instance, if $k = \bC$ and $R$ is a ring, then there is a bilinear pairing 
\begin{equation*}
    \CH^c(X \times Y)  \times  \rH^i(Y, R)  \to  \rH^{i+2c - 2\dim(Y)}(X, R),  \quad 
    (\Gamma ,  a) \mapsto \Gamma^*(a), 
\end{equation*}
defined by the formula 
\begin{equation*}
    \Gamma^*(a) = \pr_{X*}(\pr_Y^*(a) \cup \cl(\Gamma)) 
\end{equation*}
where $\cl(\Gamma) \in \rH^{2c}(X \times Y, R)$ is the cohomology class of $\Gamma$, and $\pr_{X} \colon X \times Y \to X$ and $\pr_{Y} \colon X \times Y \to Y$ are the projections. (Note that $\pr_X$ is proper by our assumption that $Y$ is proper.)  
If instead {$\Gamma \in \CH^c(X \times Y)_{\bQ}$} is only a rational correspondence, then it acts on cohomology with coefficients in any $\bQ$-algebra $R$. 
For a general field $k$, we may instead use \'{e}tale cohomology and define analogous pairings. 

\begin{remark}
\label{remark-correspondence-H}
Let us note two standard properties of the above action: 
\begin{enumerate}
\item If $f \colon X \to Y$ is a morphism and $\Gamma_f \subset X \times Y$ is its graph, then $[\Gamma_f]^* = f^*$ is the pullback on cohomology. 
If $X$ is proper and $\Gamma_f^t \subset Y \times X$ denotes the transpose of $\Gamma_f$, i.e. its image under the isomorphism $X \times Y \cong X \times Y$ swapping factors, then $[\Gamma_f^t]^* = f_*$ is the pushforward on cohomology. 

\item Let $Z$ be a smooth variety over $k$. 
If $\Gamma_1 \in \CH^*(X \times Y)$ and $\Gamma_2 \in \CH^*(Y \times Z)$, then there is a composite cycle 
\begin{equation*}
    \Gamma_2 \circ \Gamma_1 \coloneqq \pr_{XZ*}( \pr_{YZ}^*(\Gamma_2) \cdot \pr_{XY}^*(\Gamma_1)) \in \CH^*(X \times Z), 
\end{equation*}
where $\pr_{XY}$, $\pr_{YZ}$, and $\pr_{XZ}$ are the projections from $X \times Y \times Z$ onto two of the factors. (Note that $\pr_{XZ}$ is proper by our assumption that $Y$ is proper.) 
If both $Y$ and $Z$ are proper, then we have the associated pullback maps $\Gamma_1^* \colon \colon \rH^*(Y, R) \to \rH^*(X, R)$ and $\Gamma_2^* \colon \rH^*(Z, R) \to \rH^*(Y, R)$, which satisfy 
\begin{equation*}
    (\Gamma_2 \circ \Gamma_1)^* = \Gamma_1^* \circ \Gamma_2^*. 
\end{equation*}
\end{enumerate} 
\end{remark}

\subsection{Action of correspondences on Brauer groups} 
\label{section-correspondences-Br}
There is also a natural action of correspondences on Brauer groups. 
To begin, let $X$ and $Y$ be  varieties over an algebraically closed field $k$, 
with no smoothness of properness assumptions. 
Suppose $Z \subset X \times Y$ is a closed subvariety of dimension $\dim(X)$, and  
let $p \colon Z \to X$ and $q \colon Z \to Y$ denote the projections. 
We define a map 
\begin{equation*}
    Z^* \colon \Br(Y) \to \Br(k(X))
\end{equation*}
as follows:  
\begin{itemize}
    \item If $p \colon Z \to X$ is not dominant, then $Z^*(\alpha) \coloneqq 0$.
    \item If $p \colon Z \to X$ is dominant, then 
    $Z^*(\alpha) \coloneqq p_{k(X)*}(q^*(\alpha)\vert_{k(Z)})$ where we denote by $p_{k(X)} \colon \Spec(k(Z)) \to \Spec(k(X))$ the induced finite morphism. 
\end{itemize}
Extending by linearity, we obtain a bilinear pairing 
\begin{equation}
\label{equation-pairing-Br} 
    \rZ^{\dim(Y)}(X \times Y) \times \Br(Y) \to \Br(k(X)) , \quad (\Gamma, \alpha) \mapsto \Gamma^*(\alpha), 
\end{equation}
where $\rZ^{\dim(Y)}(X \times Y)$ denotes the group of codimension $\dim(Y)$ cycles on $X \times Y$. 

\begin{remark}
\label{remark-Z-surface}
If $Y$ above is a surface, 
then $Z^* \colon \Br(Y) \to \Br(k(X))$ vanishes when $q \colon Z \to Y$ is not dominant. Indeed, in this case $q$ factors via a dominant map onto $C \subset Y$ which is either a point or a curve, and therefore  $q^*(\alpha)\vert_{k(Z)}$ vanishes for any $\alpha \in \Br(Y)$ as it is the pullback of an element in $\Br(k(C)) = 0$. 
\end{remark}

\begin{remark}
If $f \colon X \to Y$ is a morphism, then $\Gamma_f^* \colon \Br(Y) \to \Br(k(X))$ is the composition of $f^* \colon \Br(Y) \to \Br(X)$ with restriction to the generic point of $X$. 
\end{remark}

The pairing~\eqref{equation-pairing-Br} descends to the level of Chow groups when restricted to unramified classes in $\Br(k(Y))$ of period prime to $\characteristic(k)$, 
and it is compatible with the action of correspondences on cohomology in the following sense. 
For any integer $n$ invertible in $k$, the Kummer sequence gives an exact sequence 
\begin{equation}
\label{kummer-seq}
    0 \to \Pic(X)/n \to \rH^2_{\et}(X, \bmu_n) \to \Br(X)[n] \to 0.
\end{equation}
In particular, when $X$ is smooth the natural map $\rH^2_{\et}(X, \bmu_n) \to \Br(k(X))$ has image $\Br(X)[n]$, 
and similarly for~$Y$. 

\begin{lemma}
\label{lemma-pairing-Br}
Let $X$ and $Y$ be smooth varieties over an algebraically closed field $k$, with $Y$ proper. Let $n$ be an integer invertible in $k$. 
Then~\eqref{equation-pairing-Br} induces a pairing  
\begin{equation*}
    \CH^{\dim(Y)}(X \times Y) \times \Br(Y)[n] \to \Br(X)[n], \quad (\Gamma, \alpha) \mapsto \Gamma^*(\alpha), 
\end{equation*}
such that the diagram 
\begin{equation*}
\xymatrix{
\CH^{\dim(Y)}(X \times Y) \times \rH_{\et}^2(Y, \bmu_n) \ar[r] \ar@{->>}[d] & \rH_{\et}^2(X, \bmu_n) \ar@{->>}[d] \\ 
   \CH^{\dim(Y)}(X \times Y) \times \Br(Y)[n] \ar[r] &  \Br(X)[n]
}
\end{equation*}
commutes, where the top horizontal arrow is given by the action of correspondences on cohomology. 
\end{lemma}

\begin{proof}
It suffices to show that for $Z \subset X \times Y$ a closed subvariety, the diagram 
\begin{equation*}
\xymatrix{
    \rH^2_{\et}(Y, \bmu_n) \ar[d] \ar[r]^{Z^*} & \rH^2_{\et}(X, \bmu_n) \ar[d] \\ 
    \Br(k(Y)) \ar[r]^{Z^*} & \Br(k(X)) 
    }
\end{equation*}
commutes. 
When $Z$ does not dominate $X$ it is clear that both compositions vanish, so we may assume $p \colon Z \to X$ is dominant. 
After base changing to the open subset $U \subset X$ given by the complement of the image of the singular locus of $Z$, we may assume that $Z$ is smooth. 
Then for $a \in \rH^2_{\et}(Y, \bmu_n)$ we have $Z^*(a) = p_*q^*(a)$, and the claim follows from the commutativity of each square in the diagram 
\begin{equation*}
\xymatrix{
    \rH^2_{\et}(Y, \bmu_n) \ar[rr]^{q^*} \ar@{=}[d] && \rH^2_{\et}(Z, \bmu_n) \ar[rr]^{p_*} \ar[d]^{(-)\vert_{k(Z)}} & & \rH^2_{\et}(X, \bmu_n) \ar[d]^{(-)\vert_{k(X)}} \\ 
    \rH^2_{\et}(Y, \bmu_n) \ar[rr]^-{q^*(-) \vert_{k(Z)}} & & \rH^2_{\et}(\Spec(k(Z)), \bmu_n) \ar[rr]^-{p_{k(X)*}} & & \rH^2_{\et}(\Spec(k(X)), \bmu_n) 
}
\end{equation*}
\end{proof} 

\subsection{The index} 
There are many different characterizations of the index of a Brauer class. Lemma~\ref{lemma-index-defs} below summarizes some of them relevant to this paper. 
The \emph{degree} of a central simple algebra $A$ over a field $K$ is the integer $\deg(A) = \sqrt{\dim_K(A)}$. 
We say that a field extension $L/K$ \emph{splits} a class $\alpha \in \Br(K)$ if $\alpha \vert_L = 0$. 
 
\begin{lemma}
\label{lemma-index-defs}
Let $K$ be a field and $\alpha \in \Br(K)$. 
Then the index $\ind(\alpha)$ is computed by any of the following: 
\begin{enumerate}
    \item \label{index-def}
    $\deg(D)$ where $D$ is the unique division algebra of class $\alpha$, 
    \item 
    \label{index-min-splitting}
    $\min \set{[L:K] \st L/K \text{ is a finite separable field extension which splits } \alpha}$, or 
    \item 
    \label{index-gcd-splitting}
    $\gcd \set{ [L: K] \st L/K \text{ is a finite separable field extension which splits } \alpha}$. 
\end{enumerate}
Moreover, if $\per(\alpha)$ is invertible in $K$, then in~\eqref{index-min-splitting} and~\eqref{index-gcd-splitting} the word ``separable'' may be omitted. 
\end{lemma} 

\begin{proof}
This is standard. 
For the equivalence of \eqref{index-def}-\eqref{index-gcd-splitting}, see for example \cite[Proposition 4.5.8 and Corollary 4.5.9]{GS}. 
If $L/K$ is a finite field extension which splits $\alpha$, then consider the separable closure $K \subset E \subset L$ of $K$ in $L$. 
Note that the restriction map  $\Br(E)[n] \to \Br(L)[n]$ on $n$-torsion is injective for $n$ coprime to $[L : E]$, as its composition with pushforward along $L/E$ is multiplication by $[L : E]$. 
It follows that if $\per(\alpha)$ is invertible in $K$ then the separable extension $E/K$ splits $\alpha$. This implies the last statement of the lemma. 
\end{proof}

For an integral scheme $X$ and $\alpha \in \Br(X)$, we define the \emph{index} of $\alpha$ as 
$\ind(\alpha) \coloneqq \ind(\alpha \vert_{k(X)})$, where $\alpha \vert_{k(X)} \in \Br(k(X))$ is the restriction to the generic point. 
This number can also be computed in terms of global objects on $X$, namely twisted sheaves. 
There are various ways to define twisted sheaves for a class $\alpha \in \rH^2_{\et}(X, \bG_m)$ on a scheme $X$, but they all depend on some auxiliary choice: 
\begin{enumerate}
    \item One can choose a $\bG_m$-gerbe $\cX \to X$ of class $\alpha$, and then consider sheaves on $\cX$ for which the inertial action is given by the standard character \cite{lieblich-moduli-twisted}. 
    \item One can choose a (hyper)covering and a cocycle $a$ representing $\alpha$, and then consider sheaves on the cover satsifying an $a$-twisted cocycle condition \cite{caldararu}. 
    \item When $\alpha$ is represented by an Azumaya algebra $\cA$, one can consider sheaves of $\cA$-modules. 
\end{enumerate} 
All of these approaches lead to equivalent theories (see \cite[\S2.1.3]{lieblich-moduli-twisted}), but each has its advantages. 
For concreteness, in the next lemma we adopt the first approach; following \cite{lieblich-moduli-twisted}, to which we refer for more details on this topic, we call an $\alpha$-twisted sheaf in this sense an \emph{$\cX$-twisted sheaf}. 

\begin{lemma}
\label{lemma-index-unramified}
Let $X$ be an integral noetherian scheme, let $\alpha \in \Br(X)$, and let $\cX \to X$ be a $\bG_m$-gerbe of class $\alpha$.  Then $\ind(\alpha)$ is computed by either of the following: 
\begin{enumerate}
    \item $\min \set{ \rk(F) \st F \text{ is a coherent $\cX$-twisted sheaf of positive rank}}$, or 
    \item $\gcd \set{ \rk(F) \st F \text{ is a coherent $\cX$-twisted sheaf of positive rank}}$.  
\end{enumerate}
\end{lemma}

\begin{proof}
For $X = \Spec(K)$ the spectrum of a field, this is \cite[Proposition 3.1.2.1]{lieblich-period-index}. 
The general case follows from this, as any $\cX_{k(X)}$-twisted coherent sheaf over the generic point can be lifted to a coherent $\cX$-twisted  sheaf by \cite[Lemma 3.1.3.1]{lieblich-period-index}. 
\end{proof}

\subsection{Period-index bounds} 
In the rest of this section, we discuss some known results about the period-index problem. 
The first is a simple observation that says the problem reduces to the case of prime power period. 
For an abelian group $G$ and a prime $p$, we write $G[p^\infty]$ for its $p$-primary part, i.e. the subgroup of $p$-power torsion elements. 

\begin{lemma}
\label{lemma-index-bound-divisible} 
Let $G \subset \Br(K)$ be a subgroup of the Brauer group of a field $K$. 
Suppose that for every prime $p$ and $\alpha \in G[p^{\infty}]$, we have 
$\ind(\alpha) \leq \per(\alpha)^e$. 
Then $\ind(\alpha)$ divides $\per(\alpha)^e$ for all $\alpha \in G$. 
\end{lemma} 

In this paper, we shall specifically be interested in the case where $K = k(X)$ is the function field of a smooth proper variety $X$ and $G = \Br(X)$.  

\begin{proof}
Recall that any torsion group is the direct sum of its $p$-primary parts. 
Let $\alpha \in G$ and let $\alpha = \alpha_1 + \alpha_2 + \cdots + \alpha_n$ be the decomposition into $p$-primary parts, so we have a factorization 
$\per(\alpha) = \prod_{i=1}^{n} p_i^{m_i}$ 
for distinct primes $p_i$ and  $\per(\alpha_i) = p_i^{m_i}$. 
It is clear that $\ind(\alpha)$ divides 
$\prod_{i=1}^{n} \ind(\alpha_i)$
(in fact these two numbers are equal \cite[Proposition 4.5.16]{GS}), so the claim reduces to the case when $\per(\alpha) = p^m$ is a prime power. 
But then $\ind(\alpha) \leq \per(\alpha)^e$ implies 
$\ind(\alpha)$ divides $\per(\alpha)^e$, because $\ind(\alpha)$ and $\per(\alpha)$ are both $p$-th powers.  
\end{proof} 

A key input for our proof of Theorem~\ref{theorem-pi-bound} is the $2$-dimensional case of the period-index conjecture, proved in 
\cite{dJ-period-index} when $\per(\alpha)$ is prime to the characteristic of $k$, 
and in \cite{lieblich-period-index, dJ-starr} in general: 

\begin{theorem}
\label{theorem-pi-surface} 
Let $K$ be a field of transcendence degree $2$ over an algebraically closed field $k$. For all $\alpha \in \Br(K)$, $\per(\alpha) = \ind(\alpha)$.  
\end{theorem} 

We will also need a result of Matzri giving a weak bound on the index in terms of the period. 
Recall that a field $K$ is called $C_r$ for an integer $r > 0$ if for every $0 < d^r < n$ and every $f \in K[x_1, \dots, x_n]$ homogeneous of degree $d$, there exists a nontrivial solution of $f$ over $K$. 
By Tsen's theorem  \cite[\href{https://stacks.math.columbia.edu/tag/03RD}{Tag 03RD}]{stacks-project}, if $K$ has transcendence degree $r$ over an algebraically closed field, then $K$ is $C_r$. 

\begin{theorem}[\cite{matzri}]
\label{theorem-matzri} 
Let $K$ be a $C_r$ field, and let $p_1, \dots, p_n$ be a collection of primes. 
Then there exists a positive integer $e$ such that for all $\alpha \in \bigoplus_{i = 1}^n \Br(K)[p_i^{\infty}]$, $\ind(\alpha)$ divides $\per(\alpha)^e$. 
Explicitly, we can take $e = p^{r-1} - 1$ where $p$ is the maximum of the $p_i$. 
\end{theorem} 

\begin{proof}
The case where the number $n = 1$ is  \cite[Theorem 6.3]{matzri}. 
The general case follows by decomposing into $p$-primary parts, as in Lemma~\ref{lemma-index-bound-divisible}. 
\end{proof}

Note that the integer $e$ guaranteed by Theorem~\ref{theorem-matzri} depends heavily on $\per(\alpha)$, and as such the result is quite far from Conjecture~\ref{conjecture-pi-bound}.


\section{Upper bounds on the index over the complex numbers} 
\label{section-upper-bound}

In this section, we prove Theorem~\ref{theorem-pi-bound}. 
Although our main goal is a result over the complex numbers, in many places we work in greater generality, for use in \S\ref{section-ub-other-fields} when we consider more general base fields. 

\subsection{Pullbacks of Brauer classes from surfaces} 
\label{section-pullbacks-surface}
As explained in \S\ref{section-correspondences-Br}, 
if $\Gamma$ is a codimension $\dim(Y)$ cycle on $X \times Y$, then it induces a pullback map $\Gamma^* \colon \Br(Y) \to \Br(k(X))$. 
Below we show that when $Y$ is a surface, then classes in the image of $\Gamma^*$ have index uniformly bounded in terms of the period; this relies on the period-index conjecture for surfaces (Theorem~\ref{theorem-pi-surface}). 

\begin{lemma}
\label{lemma-surface}
Let $k$ be an algebraically closed field, 
let $X$ be a variety over $k$, 
and let $Y$ be a surface over $k$. 
Let 
\begin{equation*}
\Gamma = \sum_{i=1}^t a_i [Z_i]  \in \rZ^{2}(X \times Y) 
\end{equation*} 
be an integeral cycle, where $a_i \in \bZ$ and $Z_i \subset X \times Y$ are $\dim(X)$-dimensional subvarieties. 
Then there exist positive integers $C$ and $N$ such that 
for all $\alpha \in \Br(Y)$ with $\per(\alpha)$ invertible in $k$, 
$\ind(\Gamma^*(\alpha))$ divides $C \cdot \per(\alpha)^{N}$. 
Moreover, $C$ and $N$ may be chosen to depend only on the set of nonzero $\deg(Z_i/X)$ counted with multiplicity. 
\end{lemma} 

\begin{proof} 
If $Z_i$ does not dominate both $X$ and $Y$, then $[Z_i]^*(\alpha) = 0$ for $\alpha \in \Br(Y)$ (see Remark~\ref{remark-Z-surface}), 
so discarding terms if necessary we may assume that all $Z_i$ dominate $X$ and $Y$. 
Let $p \colon W \to X$ be a dominant generically finite morphism of varieties which induces a Galois extension of function fields and factors via $Z_i \to X$ for all $i$, so that we have diagrams: 
\begin{equation*}
\xymatrix{ 
&& \ar[ddll]_{p} W \ar[d]_{\pi_i} \ar[ddrr]^{q} && \\  
&& Z_i \ar[dll]^{p_i} \ar[drr]_{q_i} && \\  
X &&&& Y 
}
\end{equation*} 
Note that by Lemma~\ref{lemma-push-pull}, $p_{k(X)}^*( [Z_i]^*(\alpha))  = p_{k(X)}^*(p_{i})_{k(X)*}(q_{i}^*(\alpha)\vert_{k(Z_i)})$ is a sum of $\deg(Z_i/X)$ conjugates of $q^*(\alpha) \vert_{k(W)}$ by elements $\sigma \in \Gal(k(W)/k(X))$. 
Therefore $p_{k(X)}^*(\Gamma^*(\alpha))$ is an integral combination of Galois conjugates of $q^*(\alpha) \vert_{k(W)}$; say that $\sigma_1, \dots, \sigma_N \in \Gal(k(W)/k(X))$ are the elements of the Galois group appearing in the sum. 
Up to shrinking $W$ we may assume each $\sigma_i$ is the restriction of an automorphism $g_i$ of $W$ over $X$. 

By Theorem~\ref{theorem-pi-surface} and Lemma~\ref{lemma-index-defs}, we can find a surjective generically finite morphism $\mu \colon Y' \to Y$ of degree $\per(\alpha)$ such that $\mu^*(\alpha) = 0$. 
Consider the fiber product 
\begin{equation*}
\xymatrix{
W' \ar[rrr] \ar[d]_{\nu} &&& Y' \times \cdots \times Y' \ar[d]^{\mu \times \cdots \times \mu} \\ 
W \ar[rrr]^-{(q \circ g_1, \dots, q \circ g_N)} &&& Y \times \cdots \times Y 
}
\end{equation*} 
Then $\nu \colon W' \to W$ is a dominant generically finite morphism of degree $\per(\alpha)^N$. 
Moreover, $\nu^*g_i^*q^*(\alpha)$ vanishes because it is the pullback of $\alpha$ along a map $W' \to Y$ which factors via $\mu \colon Y' \to Y$. 
By our description of $p_{k(X)}^*(\Gamma^*(\alpha))$ as a sum of Galois conjugates above, it follows that  $\nu^*_{k(W)}p_{k(X)}^*(\Gamma^*(\alpha)) = 0$, so $\Gamma^*(\alpha)$ is killed by a degree $[k(W): k(X)] \per(\alpha)^N$ extension of $k(X)$. 
Thus by Lemma~\ref{lemma-index-defs} we conclude that $\ind(\Gamma^*(\alpha))$ divides $[k(W): k(X)] \per(\alpha)^N$. 

It remains to show the claim about the parameters on which $C = [k(W): k(X)]$ and $N$ depend. 
Note that the morphism $p \colon W \to X$ can be constructed by taking a Galois closure of 
the fiber product of the $Z_i$ over $X$; 
therefore, we may choose $p \colon W \to X$ so that its degree $[k(W) : k(X)]$ divides $( \prod_{i } \deg(Z_{i}/X))!$. 
Further, by our construction, $N \leq \sum_i \deg(Z_i/X)$. 
The claim follows. 
\end{proof} 

\subsection{Absolute case}
\label{section-absolute-complex}

\begin{proposition}
\label{proposition-cycle} 
Let $X$ be a smooth proper complex variety. 
Let $i \colon Y \to X$ be an embedding of a smooth proper surface $Y$. 
Assume there exists a cycle 
\begin{equation*}
\zeta = \sum_{i=1}^t a_i [Z_i]  \in \CH^{2}(X \times Y)_{\bQ}, 
\end{equation*} 
where $a_i \in \bQ$ and $Z_i \subset X \times Y$ are $\dim(X)$-dimensional subvarieties, 
such that the induced map 
\begin{equation*}
\zeta^* \colon \rH^2(Y, \bQ) \to \rH^2(X, \bQ) 
\end{equation*}
satisfies 
$\zeta^* \circ i^* = \id_{\rH^2(X, \bQ)}$. 
Then there exists a positive integer $e$ such that for all $\alpha \in \Br(X)$, $\ind(\alpha)$ divides $\per(\alpha)^e$. 
Moreover, $e$ may be chosen to depend only on: $\dim(X)$, the orders of $\rH^2(X, \bZ)_{\mathrm{tors}}$ and $\rH^3(X, \bZ)_{\mathrm{tors}}$, the least common multiple of the denominators of the $a_i$, and the set of nonzero $\deg(Z_i/X)$ counted with multiplicity. 
\end{proposition} 

\begin{remark}
For our main results we do not need the last statement about the parameters on which $e$ depends. 
However, the corresponding statement will be relevant when we study more general base fields in \S\ref{section-ub-other-fields}; see  Proposition~\ref{proposition-cycle-k} and the proof of Theorem~\ref{theorem-pi-bound-absolute-k}.  
\end{remark}

\begin{proof}
After multiplying $\zeta$ by the least common multiple of the denominators of the $a_i$, 
we get an integral cycle $\Gamma = \sum_i b_i[Z_i] \in \CH^2(X \times Y)$. 
Up to multiplying further by a factor of $|\rH^2(X, \bZ)_{\tors}|$, we may assume that the composition 
\begin{equation*}
\rH^2(X, \bZ) \xrightarrow{\, i^* \,} \rH^2(Y, \bZ) \xrightarrow{ \, \Gamma^* \,} 
\rH^2(X, \bZ) 
\end{equation*} 
is multiplication by an integer $m > 0$. 
Note that for any integer $n > 0$ we have an exact sequence 
\begin{equation*}
0 \to \rH^2(X, \bZ)/n \to \rH^2(X, \bmu_n) \to \rH^3(X, \bZ)[n] \to 0. 
\end{equation*} 
Therefore, up to multiplying $\Gamma$ and $m$ further by a factor of $|\rH^3(X, \bZ)_{\tors}|$, we may assume that for all $n > 0$ the composition 
\begin{equation*}
\rH^2(X, \bmu_n) \xrightarrow{\, i^* \,} \rH^2(Y, \bmu_n) \xrightarrow{ \, \Gamma^* \,} 
\rH^2(X, \bmu_n) 
\end{equation*} 
is multiplication by $m$. 
It follows that $\Gamma^* \circ i^* \colon \Br(X) \to \Br(X)$ is also multiplication by $m$, see Lemma~\ref{lemma-pairing-Br}. 

Let $C$ and $N$ be integers as guaranteed by Lemma~\ref{lemma-surface} for the cycle $\Gamma$, 
so that for any $\beta \in \Br(Y)$ we have $\ind(\Gamma^*(\beta))$ divides $C \cdot \per(\beta)^N$. 
Then for any $\alpha \in \Br(X)$ we have $\ind(m \alpha)$ divides $C \cdot \per(\alpha)^N$, 
because $m\alpha = \Gamma^*(i^*(\alpha))$ as noted above. 
Thus, by Lemma~\ref{lemma-index-defs} there exists a field extension $L_{\alpha}/k(X)$ of degree dividing $C \cdot \per(\alpha)^N$ 
such that $\alpha \vert_{L_{\alpha}} \in \Br(L_{\alpha})$ has period dividing $m$. 
Since $L_{\alpha}$ is a $C_{\dim(X)}$ field, by Theorem~\ref{theorem-matzri} there is a 
further field extension $M_{\alpha}/L_{\alpha}$ of degree $d$ depending only on $\dim(X)$ and $m$ which splits $\alpha \vert_{L_{\alpha}}$. 
In summary, there is a 
field extension $M_{\alpha}/k(X)$ of degree dividing $dC \cdot \per(\alpha)^N$ which splits $\alpha \vert_{k(X)}$. 
Thus $\ind(\alpha)$ divides $dC \cdot \per(\alpha)^N$, again by Lemma~\ref{lemma-index-defs}. 

Now let $e$ be the integer given by the roundup of $N + \log_2(dC)$. 
Then for all $\alpha \in \Br(X)$, we have $\ind(\alpha) \leq \per(\alpha)^e$ . 
By Lemma~\ref{lemma-index-bound-divisible} we conclude that for all $\alpha \in \Br(X)$, 
$\ind(\alpha)$ divides $\per(\alpha)^e$. 

It remains to show the claim about the parameters on which $e$ depends. 
As noted above, $d$ only depends on $\dim(X)$ and $m$. 
In turn, by construction $m$ only depends on the least common multiple of the denominators of the $a_i$ and the orders of $\rH^2(X, \bZ)_{\mathrm{tors}}$ and $\rH^3(X, \bZ)_{\mathrm{tors}}$. 
By Lemma~\ref{lemma-surface}, $C$ and $N$ only depend on the set of nonzero $\deg(Z_i/X)$ counted with multiplicity. 
This proves the claim. 
\end{proof} 

Now we can prove the absolute case of Theorem~\ref{theorem-pi-bound}. 

\begin{theorem}
\label{theorem-pi-bound-absolute} 
Let $X$ be a smooth projective complex variety. 
Assume that the 
Lefschetz standard conjecture in degree $2$, i.e. Conjecture~\ref{conjecture-lefschetz} for $i = 2$, holds for $X$. Then Conjecture~\ref{conjecture-pi-bound} holds for $X$, i.e. there exists a positive integer~$e$ such that for all $\alpha \in \Br(X)$, 
$\ind(\alpha)$ divides $\per(\alpha)^e$. 
\end{theorem}

\begin{proof}
Let $n$ be the dimension of $X$, which we may assume is at least $3$ as the period-index conjecture is known in lower dimensions. 
Let $X \subset \bP^r$ be a closed immersion, and let $i \colon Y \to X$ be a smooth surface given as a complete intersection of $n-2$ hyperplane sections of $X$. 
If $h$ denotes the hyperplane class on $X$, then 
\begin{equation*} 
i_* \circ i^* = (-) \cup h^{n-2} \colon  \rH^2(X, \bQ) \to \rH^{2n-2}(X, \bQ). 
\end{equation*} 
By the Lefschetz standard conjecture in degree $2$, there exists a cycle $\lambda \in \CH^2(X \times X)_{\bQ}$ such that 
$\lambda^* \colon \rH^{2n-2}(X, \bQ) \to \rH^2(X, \bQ)$ is inverse to $i_* \circ i^*$. 
Therefore, 
\begin{equation*}
    \zeta \coloneqq \Gamma_i^{t} \circ \lambda \in \CH^2(X \times Y)_{\bQ}
\end{equation*}
is a cycle so that $\zeta^* \circ i^* = \id_{\rH^2(X, \bQ)}$ (see Remark~\ref{remark-correspondence-H}). Now the result follows from Proposition~\ref{proposition-cycle}. 
\end{proof}

\subsection{Relative case} 
\label{section-relative-case} 
We will prove Theorem~\ref{theorem-pi-bound} by reducing it to Theorem~\ref{theorem-pi-bound-absolute}. 
To do so, we first prove a couple preliminary results about Brauer classes in families. 

The following lemma says that a Brauer class on the fiber of a smooth proper morphism can be extended to a Brauer class on the total space, up to passing to an \'{e}tale neighborhood on the base (and assuming the period is invertible on the base). 
Here and below, we use the following notation: 
if $X \to S$ is a morphism of schemes, $\alpha \in \Br(X)$, and $s \colon \Spec(k) \to S$ is a field-valued point, we write $\alpha_s$ for the pullback of $\alpha$ to $X_s$. 

\begin{lemma}
\label{lemma-extend-brauer}
Let $f \colon X \to S$ be a smooth proper morphism of schemes. 
Let $\bar{s}$ be a geometric point of $S$ and $\alpha \in \Br(X_{\bar{s}})$, 
with $\per(\alpha)$ invertible on $S$. 
Then there exists an \'{e}tale neighborhood $(U,\bar{u}) \to (S, \bar{s})$ and a class $\tilde{\alpha} \in \Br(X_U)$ such that  $\tilde{\alpha}_{\bar{u}} = \alpha_{\bar{u}}$ and $\per(\tilde{\alpha}_{t})$ divides $\per(\alpha)$ for every $t \in U$.
\end{lemma}

\begin{proof}
Let $n = \per(\alpha)$. 
Recall from~\eqref{kummer-seq} that we have a surjection $\rH^2_{\et}(X_{\bar{s}}, \bmu_n) \twoheadrightarrow \Br(X_{\bar{s}})[n]$. 
Choose a lift $a \in \rH^2_{\et}(X_{\bar{s}}, \bmu_n)$ of $\alpha$ along this map. 
As $f \colon X \to S$ is smooth and proper and $n$ is invertible on $S$, the sheaf $\rR^2f_{*} \bmu_n$ is locally constant and its formation commutes with base change. 
Thus we can find an \'{e}tale neighborhood $(V, \bar{v}) \to (S, \bar{s})$ over which $a_{\bar{v}}$ extends to a section $\theta$ of $\rR^2f_{V*} \bmu_n$, where $f_{V} \colon X_{V} \to V$ is the base changed map and $a_{\bar{v}}$ is the pullback of $a$ to the fiber $X_{\bar{v}}$. 

Now consider the restriction map $\rH_{\et}^2(X_{V}, \bmu_n) \to \rH^0_{\et}(V, \rR^2f_{V*} \bmu_n)$. 
The obstruction to lifting $\theta$ along this map is whether it survives the Leray spectral sequence for $f_{V} \colon X_{V} \to V$. 
In other words, $\theta$ lifts to $\rH_{\et}^2(X_{V}, \bmu_n)$ if and only if 
\begin{equation*}
    d_2(\theta) = 0 \in \rH^2_{\et}(V, \rR^1f_{V*}\bmu_n) 
    \quad \text{and} 
    \quad 
    d_3(\theta) = 0 \in \rH^3_{\et}(V, \rR^0f_{V*}\bmu_n) 
\end{equation*}
where $d_2$ is the $E_2$-differential and $d_3$ is the $E_3$-differential in the Leray spectral sequence. 
A higher \'{e}tale cohomology class can be killed on a sufficiently fine \'{e}tale neighorhood of any point. Therefore, we can find an \'{e}tale neighborhood $(U,\bar{u}) \to (V, \bar{v})$ over which there is no obstruction to lifting the pullback of $\theta$ to a class $\tilde{a} \in \rH_{\et}^2(X_U, \bmu_n)$. 
The image $\tilde{\alpha}$ of $\tilde{a}$ in $\Br(X_U)[n]$ is the desired class. 
\end{proof}

Next we observe that the index of a Brauer class behaves well under specialization. 

\begin{lemma}
\label{lemma-specialize-index}
Let $X \to S$ be a smooth surjective morphism of varieties over a field, and let 
 $\alpha \in \Br(X)$. 
If $\bar{\eta}$ is the geometric generic point of $S$ and  $\bar{s}$ is a geometric point of $S$, then $\ind(\alpha_{\bar{s}})$ divides $\ind(\alpha_{\bar{\eta}})$. 
\end{lemma}

\begin{proof} 
Let $\cX \to X$ be a $\bG_m$-gerbe of class $\alpha$. 
Note that $\cX_{\bar{\eta}} \to X_{\bar{\eta}}$ and $\cX_{\bar{s}} \to X_{\bar{s}}$ are $\bG_m$-gerbes of classes $\alpha_{\bar{\eta}}$ and $\alpha_{\bar{s}}$. 
By Lemma~\ref{lemma-index-unramified} there exists a coherent $\cX_{\bar{\eta}}$-twisted sheaf of rank $r = \ind(\alpha_{\bar{\eta}})$. 
The proof proceeds by spreading out this sheaf and then specializing over $\bar{s}$. 

Namely, by spreading out we obtain a separated, finite type, dominant morphism $U \to S$ and a coherent $\cX_{U}$-twisted sheaf $F_U$ of rank $r$. 
Choose a compactification $T$ of $U \to S$ \cite[\href{https://stacks.math.columbia.edu/tag/0F41}{Tag 0F41}]{stacks-project}, i.e. a proper morphism $T \to S$ such that $U \to S$ factors via a dense open immersion $U \hookrightarrow T$. 
We may also assume that $T$ is regular by taking an alteration \cite[Theorem 4.1]{alterations}. 
The coherent $\cX_U$-twisted sheaf $F_U$ extends to a $\cX_T$-twisted sheaf $F$ (cf. \cite[Lemma 3.1.3.1]{lieblich-moduli-twisted}). 
Note that as $X_T \to T$ is smooth and $T$ is regular, $F$ is in fact a perfect $\cX_T$-twisted complex. 

The map $T \to S$ is surjective by construction, so any geometric point $\bar{s} \colon \Spec(k) \to S$ lifts to a geometric point $\bar{t} \colon \Spec(k) \to T$. 
As $F$ is perfect, its derived restriction $F_{\bar{t}}$ to $\cX_{\bar{t}} \cong \cX_{\bar{s}}$ is a perfect complex of the same rank $r$. 
Note that the rank of $F_{\bar{t}}$ can be expressed as a $\bZ$-linear combination of ranks of coherent $\cX_{\bar{t}}$-twisted sheaves. 
We conclude by Lemma~\ref{lemma-index-unramified} that $\ind(\alpha_{\bar{s}})$ divides $r$. 
\end{proof}

\begin{proof}[Proof of Theorem~\ref{theorem-pi-bound}]
If $\bar{s}$ is a geometric point of $S$ and $\alpha \in \Br(X_{\bar{s}})$, then by 
Lemmas~\ref{lemma-extend-brauer} and~\ref{lemma-specialize-index} we may find a class $\beta \in \Br(X_{\bar{\eta}})$ such that $\ind(\alpha)$ divides $\ind(\beta)$ and $\per(\beta)$ divides $\per(\alpha)$. 
On the other hand, 
the very general fiber of $X \to S$ is isomorphic to the 
geometric generic fiber, i.e. for very general $t \in S(\bC)$ there is an isomorphism $X_{t} \cong X_{\bar{\eta}}$ of schemes lying over an isomorphism $\bC \cong \overline{k(S)}$ (see e.g. \cite[Lemma 2.1]{vial}). 
Therefore, the result follows from Theorem~\ref{theorem-pi-bound-absolute}. 
\end{proof} 


\section{Upper bounds on the index over other fields} 
\label{section-ub-other-fields}  
In this section we explain how to extend Theorem~\ref{theorem-pi-bound} to arbitrary algebraically closed base fields. 
The main results are Theorem~\ref{theorem-pi-bound-absolute-k} and Theorem~\ref{theorem-pi-bound-relative-k}. 

\subsection{Lefschetz standard conjecture} 
Similar to Conjecture~\ref{conjecture-lefschetz}, the Lefschetz standard conjecture can be formulated for any suitable Weil cohomology theory for varieties over $k$. 
We consider $\ell$-adic  cohomology, as this is what will be used below. 

The setup is as follows. 
Let $X$ be a smooth projective variety  of dimension $n$ over an algebraically closed field $k$, with $h \in \CH^1(X)$ the first Chern class of an ample line bundle. 
For any prime $\ell \neq \characteristic(k)$ and $0 \leq i \leq n$, the hard Lefschetz theorem for \'{e}tale cohomology says that the map 
\begin{equation*}
    L_{\ell}^{n-i} \coloneqq (-) \cup h^{n-i} \colon \rH^i_{\et}(X, \bQ_{\ell}) \to \rH^{2n-i}_{\et}(X, \bQ_{\ell})
\end{equation*}
is an isomorphism. 

\begin{conjecture}[$\ell$-independent Lefschetz standard conjecture in degree $i$]
\label{conjecture-l-lefschetz}
For $X$, $h$, and $i$ as above, there exists a codimension $i$ cycle $\lambda \in \CH^i(X \times X)_{\bQ}$ such that for every prime $\ell \neq \characteristic(k)$ the induced map 
\begin{equation*}
    \lambda^* \colon \rH_{\et}^{2n-i}(X, \bQ_{\ell}) \to \rH_{\et}^{i}(X, \bQ_{\ell}) 
\end{equation*}
is the inverse of $L_{\ell}^{n-i}$.  
\end{conjecture}

\begin{remark}
\label{remark-l-lefschetz}
We warn the reader that 
``the Lefschetz standard conjecture in degree $i$'' sometimes refers to different statements in the literature. 
We have incorporated ``independence of $\ell$'' in our formulation, which is nonstandard but useful for our purposes. 
(In fact, for our purposes it would be sufficient to know the statement for all but finitely many $\ell$ --- see  Remark~\ref{remark-almost-all-l-suffice} --- but we did not formulate the conjecture in this way as it seemed artificial.)

In some cases, independence of $\ell$ comes for free. 
For instance, when $k = \bC$, the comparison isomorphism with singular cohomology shows that Conjecture~\ref{conjecture-l-lefschetz} coincides with Conjecture~\ref{conjecture-lefschetz}.
Similarly, when $\characteristic(k) = 0$, it follows that the assertion of the Conjecture~\ref{conjecture-l-lefschetz} for a single $\ell$ implies it for all $\ell$. 
When $\characteristic(k) > 0$, such an independence of $\ell$ statement is unknown, but it would follow from the Hodge standard conjecture on $X \times X$, cf. \cite[Proposition 5.4]{kleiman-standard}. 
\end{remark}

\subsection{Absolute case}

\begin{proposition}
\label{proposition-cycle-k} 
Let $X$ be a smooth proper variety over an algebraically closed field $k$. 
Let $i \colon Y \to X$ be an embedding of a smooth proper surface $Y$. 
Let $\ell \neq \characteristic(k)$ be a prime. 
Assume there exists a cycle 
\begin{equation*}
\zeta = \sum_{i=1}^t a_i [Z_i]  \in \CH^{2}(X \times Y)_{\bQ}, 
\end{equation*} 
where $a_i \in \bQ$ and $Z_i \subset X \times Y$ are $\dim(X)$-dimensional subvarieties, 
such that the induced map 
\begin{equation*}
\zeta^* \colon \rH^2_{\et}(Y, \bQ_{\ell}) \to \rH^2_{\et}(X, \bQ_{\ell}) 
\end{equation*}
satisfies $\zeta^* \circ i^* = \id_{\rH^2_{\et}(X, \bQ_{\ell})}$. 
Then there exists a positive integer $e$ such that for all $\ell$-primary classes $\alpha \in \Br(X)[\ell^{\infty}]$,  
$\ind(\alpha)$ divides $\per(\alpha)^e$. 
Moreover, $e$ may be chosen to depend only on: $\dim(X)$, the orders of $\rH^2_{\et}(X, \bZ_{\ell})_{\mathrm{tors}}$ and $\rH^3_{\et}(X, \bZ_{\ell})_{\mathrm{tors}}$, the least common multiple of the denominators of the $a_i$, and the set of nonzero $\deg(Z_i/X)$ counted with multiplicity. 
\end{proposition} 

\begin{proof}
The proof is identical to that of Proposition~\ref{proposition-cycle}, up to replacing singular cohomology with $\bZ$-coefficients by \'{e}tale cohomology with $\bZ_{\ell}$-coefficients. 
\end{proof}

Now we prove the absolute version of Theorem~\ref{theorem-pi-bound} in our context. 
\begin{theorem}
\label{theorem-pi-bound-absolute-k} 
Let $X$ be a smooth projective variety over an algebraically closed field $k$. 
Assume that the $\ell$-independent
Lefschetz standard conjecture in degree $2$, i.e. Conjecture~\ref{conjecture-l-lefschetz} for $i = 2$, holds for $X$. Then Conjecture~\ref{conjecture-pi-bound} holds for $X$, i.e. there exists a positive integer~$e$ such that for all $\alpha \in \Br(X)$, 
$\ind(\alpha)$ divides $\per(\alpha)^e$. 
\end{theorem}

\begin{remark}
\label{remark-almost-all-l-suffice}
As the proof below shows, at the expense of needing a possibly larger value of $e$, the same conclusion holds if instead of assuming the $\ell$-independent Lefschetz standard conjecture in degree $2$ as formulated in Conjecture~\ref{conjecture-l-lefschetz}, we assume the weaker variant where $\lambda^*$ is only required to be inverse to $L^{n-2}_{\ell}$ for \emph{almost all} $\ell$. 
\end{remark}

\begin{proof}
We follow the proof of Theorem~\ref{theorem-pi-bound-absolute} with some modifications. 
Let $n$ be the dimension of $X$, which we may assume is at least $3$. 
Let $X \subset \bP^r$ be a closed immersion, and let $i \colon Y \to X$ be a smooth surface given as a complete intersection of $n-2$ hyperplane sections of $X$. 
If $h$ denotes the hyperplane class on $X$, then for any prime $\ell \neq \characteristic(k)$ we have 
\begin{equation*}
    i_* \circ i^* = (-) \cup h^{n-2} \colon  \rH^2_{\et}(X, \bQ_{\ell}) \to \rH^{2n-2}_{\et}(X, \bQ_{\ell}). 
\end{equation*}
By the Lefschetz standard conjecture in degree $2$, there exists a cycle $\lambda \in \CH^2(X \times X)_{\bQ}$ such that 
$\lambda^* \colon \rH^{2n-2}_{\et}(X, \bQ_{\ell}) \to \rH^2_{\et}(X, \bQ_{\ell})$ is inverse to $i_* \circ i^*$ for all $\ell \neq \characteristic(k)$. 
Therefore, 
\begin{equation*}
    \zeta \coloneqq \Gamma_i^{t} \circ \lambda \in \CH^2(X \times Y)_{\bQ}
\end{equation*}
is a cycle so that $\zeta^* \circ i^* = \id_{\rH^2_{\et}(X, \bQ_{\ell})}$ for $\ell \neq \characteristic(k)$ (see Remark~\ref{remark-correspondence-H}). 
Thus by Proposition~\ref{proposition-cycle-k} there exists a positive integer $e_{\ell}$ such that $\ind(\alpha)$ divides $\per(\alpha)^{e_{\ell}}$ for all $\alpha \in \Br(X)[\ell^{\infty}]$, 
where $e_{\ell}$ only depends on the orders of $\rH^2_{\et}(X, \bZ_{\ell})_{\mathrm{tors}}$ and $\rH^3_{\et}(X, \bZ_{\ell})_{\mathrm{tors}}$ and other data that does not depend on $\ell$ (namely, $\dim(X)$ and numbers associated to the expression of $\zeta$ as a $\bQ$-linear combination of classes of subvarieties). 

By \cite{gabber-torsion} for almost all $\ell$ the cohomology $\rH^*_{\et}(X, \bZ_{\ell})$ is torsion free, so we may find a single positive integer $f$ such that for all primes $\ell \neq \characteristic(k)$ and $\alpha \in \Br(X)[\ell^{\infty}]$,  $\ind(\alpha)$ divides $\per(\alpha)^{f}$. 
On the other hand, at the prime $p = \characteristic(k)$, Theorem~\ref{theorem-matzri} gives an integer $f_p$ (explicitly, $f_p = p^{\dim(X) - 1}-1$) such that for all $\alpha \in \Br(X)[p^{\infty}]$, $\ind(\alpha)$ divides $\per(\alpha)^{f_p}$. 
Thus setting $e = \max\set{f,f_p}$, we conclude by Lemma~\ref{lemma-index-bound-divisible} that for all $\alpha \in \Br(X)$, $\ind(\alpha)$ divides $\per(\alpha)^e$. 
\end{proof}

\subsection{Relative case} 
Theorem~\ref{theorem-pi-bound-absolute-k} implies the following relative statement. 

\begin{theorem}
\label{theorem-pi-bound-relative-k}
Let $X \to S$ be a smooth proper morphism of varieties over an algebraically closed field.
Assume that the geometric generic fiber of $X \to S$ is projective and satisfies the $\ell$-independent Lefschetz standard conjecture in degree $2$.  
Then there exists a positive integer~$e$ such that for all geometric points $\bar{s}$ of $S$ and $\alpha \in \Br(X_{\bar{s}})$, 
$\ind(\alpha)$ divides $\per(\alpha)^e$. 
\end{theorem}

\begin{proof}
The same argument as in our proof of Theorem~\ref{theorem-pi-bound} in \S\ref{section-relative-case} provides an $f$ such that for all geometric points $\bar{s}$ of $S$ and $\alpha \in \Br(X_{\bar{s}})$ with $\per(\alpha)$ prime to $p = \characteristic(k)$, 
$\ind(\alpha)$ divides $\per(\alpha)^{f}$. 
Then if $r$ is the relative dimension of $X \to S$, the integer $e = \max \set{ f, p^{r-1} - 1}$ satisfies the conclusion of the theorem, by the same argument as at the end of the proof of Theorem~\ref{theorem-pi-bound-absolute-k}. 
\end{proof}


\section{Lower bounds on the index over the complex numbers} 
\label{section-lower-bound}
In this section, we prove Theorem~\ref{theorem-lower-bound} in \S\ref{section-obstruction}, after some preparatory results on Severi--Brauer varieties. 
Then we discuss applications to potential period-index conjecture obstructions in \S\ref{section-obstruction-PI}, the construction of Brauer classes of large index in \S\ref{section-large-index}, and counterexamples to the integral Hodge conjecture in \S\ref{section-IHC}. 
In \S\ref{section-index-via-0-cycles} we discuss an auxiliary divisibility obstruction for the index of Brauer classes, which is used in \S\ref{section-IHC}.

\subsection{The index in terms of Severi--Brauer varieties} 
We start by reviewing some standard results about Severi--Brauer varieties; see for instance \cite{GS, kollar} for more details. 

Let $K$ be a field. 
A \emph{Severi--Brauer variety} over $K$ is a twisted form of a projective space, i.e. a variety $P$ over $K$ such that $P_{\bar{K}} \cong \bP_{\bar{K}}^{\dim(P)}$. 
The set of Severi--Brauer varieties over $K$ of dimension $r$ up to isomorphism is in canonical bijection with $\rH^1(K, \PGL_{r+1})$, and hence also with the set of central simple algebras over $K$ of degree $r+1$ up to isomorphism. In particular, for any Severi--Brauer variety $P$, there is an associated class $[P] \in \Br(K)$, and any $\alpha \in \Br(K)$ arises (non-uniquely) in this way. 

A \emph{twisted linear subvariety} of $P$ 
is a closed subvariety $L \subset P$ such that over the algebraic closure $L_{\bar{K}} \subset P_{\bar{K}} \cong \bP_{\bar{K}}^{\dim(P)}$ is a nonempty linear subspace. 
Note that $L$ is then itself a Severi--Brauer variety. 
Under the correspondence with central simple algebras, the existence of a twisted linear embedding corresponds to taking matrix algebras: if $A$ and $B$ are central simple algebras corresponding to Severi--Brauer varieties $P$ and $Q$, then there is an isomorphism $\rM_{r}(A) \cong B$ for some $r$ if and only if there is an embedding $P \hookrightarrow Q$ as a twisted linear subvariety. 
This leads to a description of the dimensions of twisted linear subvarieties of a Severi--Brauer variety, cf. \cite[Lemma 27]{kollar}. 

\begin{lemma}
\label{lemma-twisted-linear}
Let $m$ be the minimal dimension of a twisted linear subvariety of $P$. 
Then $m+1$ divides $\dim(L) + 1$ for any twisted linear subvariety $L \subset P$, and there exists a twisted linear subvariety $L \subset P$ with $\dim(L) + 1= i(m+1)$ for any $1 \leq i \leq \frac{\dim(P)+1}{m+1}$. 
\end{lemma}

The index of a Brauer class can be computed in terms of the geometry of Severi--Brauer varieties. 

\begin{lemma}
\label{lemma-index-SB}
Let $K$ be a field and $\alpha \in \Br(K)$. 
Then $\ind(\alpha)$ is computed by 
\begin{enumerate}
    \item $\min \set{ \dim(P) + 1 \st P \text{ is a Severi--Brauer variety of class } \alpha }$. 
\end{enumerate}
Further, if $P$ is any Severi--Brauer variety of class $\alpha$, then $\ind(\alpha)$ is computed by any of the following: 
\begin{enumerate}[resume]
    \item \label{index-0-cycle} 
    $\min \set{ \deg(Z) \st Z \text{ is a $0$-cycle of positive degree on } P}$,   
    \item $\min \set{ \dim(L)+1 \st \text{$L \subset P$ is a twisted linear subvariety of $P$}}$, or 
    \item $\min \set{ \codim_P(L) \st \text{$L \subsetneq P$ is a proper twisted linear subvariety of $P$}}$. 
\end{enumerate}
One may also replace $\min$ with $\gcd$ in any of the above descriptions. 
\end{lemma}

\begin{proof}
See for instance \cite[\S3 and \S6]{kollar}. 
\end{proof}

There is also a relative notion of a Severi--Brauer variety. 
If $X$ is a scheme, a Severi--Brauer scheme over $X$ is a morphism of schemes $P \to X$ which is locally in the \'{e}tale topology on $X$ isomorphic to a projective space $\bP_{X}^{r}$ over $X$. 
The set of Severi--Brauer schemes over $X$ of relative dimension $r$ up to isomorphism 
is in canonical bijection with $\rH^1_{\et}(X, \PGL_{r+1})$, and hence also with the set of Azumaya algebras over $X$ of degree $r+1$ up to isomorphism. 
In particular, for any Severi--Brauer scheme over $X$, there is an associated class $[P] \in \Br(X)$. 

\begin{remark}
\label{remark-Oalpha1}
If $\alpha \in \Br(X)$ and $\pi \colon P \to X$ is a Severi--Brauer scheme of class $\alpha$, 
then there is a $\pi^*\alpha$-twisted line bundle $\cO_{P}^{\alpha}(1)$ on $P$, whose restriction to any geometric fiber $P_{\bar{x}} \cong \bP^{r}$ is identified with the standard $\cO_{\bP^{r}}(1)$ line bundle. 
Here and below, we use the C\u{a}ld\u{a}raru model for twisted sheaves, whereby an $\alpha$-twisted sheaf is defined in terms of a choice (usually suppressed in notation) of cocycle for $\alpha$, cf. the discussion preceding Lemma~\ref{lemma-index-unramified}. 
\end{remark} 

\subsection{Topologically trivial Brauer classes}
Let $X$ be a complex variety. 
We write $\Brtop(X)$ for the \emph{topological Brauer group} of $X$, i.e. the Brauer group of the topological space underlying the analytification $X^{\an}$ of $X$. 
See \cite[I, \S1]{grothendieck-brauer} for background on Brauer groups of topological spaces, which are defined analogously to the case of schemes; note that by \cite[I, Th\'{e}or\`{e}me~1.6]{grothendieck-brauer}, the ``Azumaya'' or ``cohomological'' definitions of the topological Brauer group are equivalent. 
Similar to the case of schemes, for $\alpha \in \Brtop(X)$ we can consider $\alpha$-twisted sheaves on the topological space $X^{\an}$; see for instance \cite{lieblich-moduli-twisted} where the theory of twisted sheaves is set up on any ringed topos. 

\begin{definition}
For a complex variety $X$ and $\alpha \in \Br(X)$, we say that $\alpha$ is \emph{topologically trivial} if its image under the restriction map $\Br(X) \to \Brtop(X)$ vanishes. 
\end{definition}

\begin{lemma}
\label{lemma-top-triv} 
Let $X$ be a complex variety and $\alpha \in \Br(X)$. The following are equivalent: 
\begin{enumerate}
    \item \label{top-triv-1}
    The class $\alpha$ is topologically trivial. 

    \item \label{top-triv-3}
    The image of $\alpha$ under the composition 
    \begin{equation}
    \label{Br-Brtop-restriction}
 \Br(X) = \rH^2_{\et}(X, \bG_m)_{\tors} \to \rH^2(X^{\an}, \bG_m)_{\tors} \to \rH^3(X, \bZ)_{\tors} 
    \end{equation}
    vanishes, where the last map comes from the exponential sequence. 
    \item \label{top-triv-4}
    If $\alpha \in \Br(X)[n]$, there exists 
    a class $b \in \rH^2(X, \bZ)$ whose image under the map 
    \begin{equation*}
        \rH^2(X, \bZ) \xrightarrow{\exp(-/n)} 
        \rH^2(X^{\an}, \bmu_n) \cong \rH^2_{\et}(X, \bmu_n) \to \Br(X)[n] 
    \end{equation*}
    is equal to $\alpha$, where we write $\exp(-)$ for the usual exponential applied to $2\pi i(-)$.\footnote{We use this convention because we want to ease notation by suppressing Tate twists below.}
    \item \label{top-triv-5}
    If $\alpha \in \Br(X)[n]$, 
    there exists a lift $\theta \in \rH^2(X^{\an}, \bmu_n)$ of $\alpha \in \Br(X)$ and a $\theta$-twisted topological line bundle on $X^{\an}$. 
    In particular, there exists an $\alpha$-twisted topological line bundle on $X^{\an}$. 
\end{enumerate}
Further, if $\alpha \in \Br_{\Az}(X)$, so that it is represented by a Severi--Brauer variety, then the above conditions are also equivalent to: 
\begin{enumerate}[resume]
    \item \label{top-triv-2}
    For any Severi--Brauer variety $P \to X$ of class $\alpha$, there exists a topological complex vector bundle $E$ on $X^{\an}$ such that there is a homeomorphism $P \cong \bP_{X^{\an}}(E)$ of topological spaces. 
\end{enumerate}
\end{lemma}

\begin{proof}
By \cite[Corollaire~1.7]{grothendieck-brauer}  there is an isomorphism $\Brtop(X) \cong \rH^3(X, \bZ)_{\tors}$ under which $\Br(X) \to \Brtop(X)$ is identified with the map~\eqref{Br-Brtop-restriction}. 
This gives the equivalence of~\eqref{top-triv-1} and~\eqref{top-triv-3}. 

The equivalence of~\eqref{top-triv-3} and \eqref{top-triv-4} follows from the commutative diagram 
\begin{equation*}
\xymatrix{
\rH^2(X, \bZ) \ar[rr]^{\exp(-/n)} \ar[d]_{(-)/n} && \rH^2(X^{\an}, \bmu_n) \ar[rr] \ar[d] && \rH^3(X, \bZ)  \ar@{=}[d]  \\ 
\rH^2(X^{\an}, \cO_{X^{\an}}) \ar[rr]^{\exp} && \rH^2(X^{\an}, \bG_m) \ar[rr] && \rH^3(X, \bZ) 
}
\end{equation*} 
with exact rows, since $\rH^{2}(X^{\an}, \bmu_n)$ surjects onto $\Br(X)[n]$. 

Given a $\theta$-twisted topological line bundle $L$ on $X^{\an}$ as in \eqref{top-triv-5}, $L^{\otimes n}$ is an honest topological line bundle on $X^{\an}$. Then $b = c_1(L^{\otimes n}) \in \rH^2(X, \bZ)$ satisfies $\exp(b/n) = \theta \in \rH^2(X^{\an}, \bmu_n)$; indeed, this follows from (the topological analog of) \cite[Proposition 2.3.4.4]{lieblich-moduli-twisted}. 
Hence $b$ is a class as required in \eqref{top-triv-4}. 

To show~\eqref{top-triv-4} implies~\eqref{top-triv-5}, we first make some preliminary observations. Let 
\begin{equation*}
\Pic^{\mathrm{top}}(X) \coloneqq \rH^1(X^{\an}, (\cO_X^{\mathrm{top}})^{\times})
\end{equation*} 
denote the topological Picard group, where 
$\cO_X^{\mathrm{top}}$ is the sheaf of continuous $\bC$-valued functions on $X^{\an}$. 
There is an isomorphism 
\begin{equation*}
\Pic^{\mathrm{top}}(X) \xrightarrow{ \, c_1 \, } \rH^2(X, \bZ) 
\end{equation*} 
given by the first chern class of a topological line bundle $L$ on $X^{\an}$ (which can also be described as the boundary map for the topological exponential sequence). The topological Kummer sequence gives an exact sequence 
\begin{equation*}
\Pic^{\mathrm{top}}(X) \xrightarrow{ \, n \, } \Pic^{\mathrm{top}}(X) \xrightarrow{\delta} \rH^2(X^{\an}, \bmu_n). 
\end{equation*} 
Given $M \in \Pic^{\mathrm{top}}(X)$ there may not exist an $n$-th root, but the obstruction is measured by 
$\delta(M) \in \rH^2(X^{\an}, \bmu_n)$ and there does exist a $\delta(M)$-twisted line bundle $L$ such that $L^{\otimes n} \cong M$. 
Now given $b \in \rH^2(X, \bZ)$ as in \eqref{top-triv-4}, if we take $M$ to be the topological line bundle with $c_1(M) = b$ and $\theta = \delta(M)$, then the $n$-th root $L$ is a $\theta$-twisted topological line bundle as required by \eqref{top-triv-5}. 

Finally, the equivalence of~\eqref{top-triv-1} and~\eqref{top-triv-2} follows directly from the definitions.
\end{proof}

\begin{remark}
\label{remark-rational-B}
When $X$ is proper, then the analytic and algebraic Brauer groups coincide, i.e. the canonical map $\rH^2_{\et}(X, \bG_m)_{\tors} \to \rH^2(X^{\an}, \bG_m)_{\tors}$ is an isomorphism (see e.g. \cite[Proposition 1.3]{schroer}). 
Therefore, we have a well-defined map 
\begin{equation}
\label{exp} 
    \exp \colon \rH^2(X, \bQ) \to \Br(X). 
\end{equation}
Taking $B = b/n$, it is clear that condition~\eqref{top-triv-5} in Lemma~\ref{lemma-top-triv} is equivalent to: 
\begin{enumerate}\addtocounter{enumi}{5}
    \item  \label{top-triv-6}
    There exists a class $B \in \rH^2(X, \bQ)$ such that $\exp(B) = \alpha$. 
\end{enumerate}
\end{remark}

\begin{definition}
If $X$ is a complex variety and $\alpha \in \Br(X)$ is topologically trivial, then we call a class $b \in \rH^2(X, \bZ)$ as in~\eqref{top-triv-4} of Lemma~\ref{lemma-top-triv} an \emph{integral $B$-field of degree $n$} for $\alpha$. 
When $X$ is proper, we call $B \in \rH^2(X, \bQ)$ as in~\eqref{top-triv-6} above a \emph{rational $B$-field} for $\alpha$. 
\end{definition}

\begin{remark} 
\label{remark-top-O1}
If $P \to X$ is a Severi--Brauer variety whose class $\alpha \in \Br(X)$ is topologically trivial,  then by Lemma~\ref{lemma-top-triv}\eqref{top-triv-2} the map $\pi \colon P \to X$ is topologically a projective bundle, so there is a topological relative $\cO(1)$ line bundle. Explicitly, let $L$ be an $\alpha$-twisted topological line bundle on $X^{\an}$, which exists by Lemma~\ref{lemma-top-triv}\eqref{top-triv-5}; then 
\begin{equation} 
\label{Otop1}
\cO_{P}^{\mathrm{top}}(1) \coloneqq \cO_{P}^{\alpha}(1) \otimes \pi^*L^{-1}
\end{equation} 
is an (untwisted) topological line bundle on $P$, whose restriction to fibers $P_{x} \cong \bP^r$ is $\cO_{\bP^r}(1)$ as a topological line bundle. 
When we are given an integral $B$-field $b$ of degree $n$ for $\alpha$, we will always define $\cO_{P}^{\mathrm{top}}(1)$ using an $\alpha$-twisted topological line bundle $L$ corresponding to $b$ via Lemma~\ref{lemma-top-triv}, which satisfies $c_1(L^{\otimes n}) = b$. 
\end{remark}

\begin{lemma}
\label{lemma-cohomology-SB}
Let $X$ be a smooth proper complex variety. 
Let $\alpha \in \Br(X)$ be a topologically trivial class, with $b$ an integral $B$-field of degree $n$ for $\alpha$ and $B = b/n$ the corresponding rational $B$-field. 
Let $\pi \colon P \to X$ be a Severi--Brauer variety of class $\alpha$ and relative dimension $r$, with $h \in \rH^2(P, \bZ)$ the first Chern class of the topological line bundle $\cO^{\mathrm{top}}(1)$ from Remark~\ref{remark-top-O1}. 
\begin{enumerate}
\item For any integer $d$, there is an isomorphism of abelian groups 
\begin{equation}
\label{equation-H-P}
    \sum_j \pi^*(-) \cup h^j \colon \bigoplus_{j = 0}^r \rH^{d - 2j}(X, \bZ)  \xrightarrow{\, \sim \,} \rH^d(P, \bZ). 
\end{equation}
\item \label{nhb-algebraic}
The class $nh + \pi^*b \in \rH^2(P, \bZ)$ is algebraic. 

\item 
There is an isomorphism of rational Hodge structures 
\begin{equation}
\label{equation-Hodge-P}
\sum_j \pi^*(-) \cup (h+\pi^*B)^{j} \colon 
\bigoplus_{j=0}^r \rH^{d-2j}(X, \bQ)(-j) \xrightarrow{\, \sim \,} \rH^d(P, \bQ) . 
\end{equation}
\end{enumerate}
\end{lemma}

\begin{proof}
The isomorphism~\eqref{equation-H-P} is the standard formula for the cohomology of a topological projective bundle. 
For \eqref{nhb-algebraic}, note that $\cO_{P}^{\alpha}(1)^{\otimes n}$ is an (untwisted) line bundle on $P$; hence its Chern class, which by~\eqref{Otop1} equals $nh + \pi^*c_1(L^{ \otimes n}) = nh + \pi^*b$, is algebraic. 
This implies that~\eqref{equation-Hodge-P} is a map of Hodge structures, while the fact that it is an isomorphism follows from~\eqref{equation-H-P}. 
\end{proof}

\subsection{Divisibility obstructions}
\label{section-obstruction}

The following result is a more precise form of Theorem~\ref{theorem-lower-bound}. 

\begin{theorem}
\label{theorem-divisibility} 
Let $X$ be a smooth projective complex variety with $\alpha \in \Br(X)$. 
Let $e$ be a positive integer and 
$\pi \colon P \to X$ a Severi--Brauer variety of class $\alpha$ and relative dimension at least $e$.  
\begin{enumerate} 
\item \label{divisible-1}
$\ind(\alpha)$ divides $e$ if and only if there exists a closed subvariety $Z \subset P$ whose generic fiber $Z_{\eta} \subset P_{\eta}$ is a twisted linear subvariety of codimension $e$. 
In particular, if $\ind(\alpha)$ divides $e$, there exists an integral Hodge class $\gamma \in \rH^{e,e}(P, \bZ)$ whose restriction to fibers is the class of a codimension $e$ linear subspace. 
\item \label{divisible-2}
If $\alpha$ is topologically trivial and $B \in \rH^2(X, \bQ)$ is a rational $B$-field for $\alpha$, then there exists a class $\gamma \in \rH^{e,e}(P, \bZ)$ satisfying the above condition if and only if there exist Hodge classes $c_j \in \rH^{j,j}(X, \bQ)$ as in the conclusion of Theorem~\ref{theorem-lower-bound}. 
Explicitly, $\gamma$ and the $c_j$ are related by  
\begin{align*}
    \gamma & = (h+\pi^*(B))^e + \pi^*(c_1)(h + \pi^*(B))^{e-1} + \cdots + \pi^*(c_m)(h+\pi^*(B))^{e-m} \\ 
    & = h^e + \pi^*(p_1^{B,e}(c_1))h^{e-1} + 
    \pi^*(p_2^{B,e}(c_1,c_2))h^{e-2} + \cdots + \pi^*(p_{m}^{B,e}(c_1, \dots, c_m)) h^{e - m} , 
\end{align*} 
where $m = \min \set{ e, \dim(X) }$. 
\end{enumerate} 
\end{theorem}

\begin{proof}
\eqref{divisible-1} follows from Lemmas~\ref{lemma-index-SB} and~\ref{lemma-twisted-linear} by taking the closure of a twisted linear subvariety of the generic fiber. 
For~\eqref{divisible-2}, suppose that $\gamma \in \rH^{e,e}(P, \bZ)$ is an integral Hodge class whose restriction to fibers is the class of a codimension $e$ linear subspace. 
Then in terms of the decomposition~\eqref{equation-H-P}, we can write 
\begin{equation}
\label{gamma-1}
    \gamma = h^e + \pi^*(b_1) h^{e-1} + \pi^*(b_2)h^{e-2} + 
    \cdots + \pi^*(b_e), 
\end{equation}
where $b_j \in \rH^{2j}(X, \bZ)$. 
On the other hand, by the description of the Hodge structure of $P$ in~\eqref{equation-Hodge-P}, in rational cohomology we can write 
\begin{equation}
\label{gamma-2}
    \gamma = (h+\pi^*(B))^e + \pi^*(c_1)(h + \pi^*(B))^{e-1} + \pi^*(c_2)(h + \pi^*(B))^{e-2} + \cdots + \pi^*(c_e), 
\end{equation}
where $c_j \in \rH^{j,j}(X, \bQ)$. 
Expanding~\eqref{gamma-2} in terms of powers of $h$ and comparing with~\eqref{gamma-1}, we find that the $c_j$ satisfy the conclusion of Theorem~\ref{theorem-lower-bound}. 
Reversing the logic, given such $c_j$, defining $\gamma$ by~\eqref{gamma-2} gives an integral Hodge class of the required form. 
\end{proof}

\begin{example}[Enhanced Kresch obstruction]
\label{example-kresch}
Let $X$ be a smooth projective complex variety, let $b \in \rH^2(X, \bZ)$, let $B = b/2 \in \rH^2(X, \bQ)$, and let $\alpha = \exp(B) \in \Br(X)[2]$ be the corresponding topologically trivial class. 
If $\ind(\alpha) = 2$, then by  Theorem~\ref{theorem-lower-bound} there exist Hodge classes $c_1 \in \rH^{1,1}(X, \bQ)$ and $c_2 \in \rH^{2,2}(X, \bQ)$ such that 
\begin{align*}
p_{1}^{B,2}(c_1) & = b + c_1, \\ 
p_{2}^{B,2}(c_1, c_2) & = \frac{b^2}{4} + \frac{b c_1}{2} + c_2 ,
\end{align*} 
are both integral. The first says that $c_1$ is integral, and multiplying the second by $4$ shows that 
\begin{equation*}
b^2  + 2bc_1 + 4c_2 \in 4\rH^4(X, \bZ), 
\end{equation*}
and hence $4c_2$ is integral. 
In sum, if $\ind(\alpha) = 2$ then there exist integral Hodge classes 
$c \in \rH^{1,1}(X, \bZ)$ and 
$d \in \rH^{2,2}(X, \bZ)$ such that
\begin{equation}
\label{kresch-enhanced} 
        b^2 \equiv 2bc + d \pmod 4.
\end{equation}
Modulo $2$ this gives 
\begin{equation}
\label{kresch-original}
    b^2 \equiv d \pmod 2,  
\end{equation}
which is precisely Kresch's obstruction to $\ind(\alpha) = 2$ from \cite{kresch}, where examples with nontrivial obstruction are also constructed. 
We note that the obstruction~\eqref{kresch-enhanced} is stronger than and leads to interesting examples inaccessible by~\eqref{kresch-original}; see for instance the $g = 3$ case of Proposition~\ref{proposition-index-av}. 
\end{example}

As explained in the introduction, Theorem~\ref{theorem-divisibility} is inspired by Hotchkiss's Hodge-theoretic index $\ind_{\Hdg}(\alpha)$, which must divide $e$ in order for $\ind(\alpha)$ to do so. 
The following shows that Hotchkiss's obstruction is at least as strong as ours: 

\begin{lemma}
\label{lemma-obstruction-comparison}
Let $X$ be a smooth projective complex variety with $\alpha \in \Br(X)$. 
Let 
$e$ be a positive integer and 
$\pi \colon P \to X$ a Severi--Brauer variety of class $\alpha$ and relative dimension at least $e$.  
If $\ind_{\Hdg}(\alpha)$ divides $e$, then there exists an integral Hodge class $\gamma \in \rH^{e,e}(P, \bZ)$ whose restriction to fibers is the class of a codimension $e$ linear subspace.  
\end{lemma}

\begin{proof} 
By definition there is a Hodge class in $\Ktop[0](X, \alpha)$ of rank $\ind_{\Hdg}(\alpha)$, so by taking an appropriate multiple we find a Hodge class $E \in \Ktop[0](X, \alpha)$ of rank $e$. 
We consider the class $F = \pi^*(E^{\vee}) \otimes \cO_{P}^{\alpha}(1)$ 
which lies in $\Ktop[0](P)$ (without a twist) and is Hodge (because multiplication by the algebraic $\alpha$-twisted line bundle $\cO_{P}^{\alpha}(1)$ preserves being Hodge). 
Then $\gamma = c_e(F) \in \rH^{e,e}(P, \bZ)$ is an integral Hodge class, and if $x \in X(\bC)$ then the restriction of $\gamma$ to $P_x \cong \bP^{r}$ is equal to $c_e( \cO_{\bP^r}(1)^{\oplus e})$, i.e. the class of a codimension $e$ linear subspace.
\end{proof}

\subsection{Period-index conjecture obstructions} 
\label{section-obstruction-PI} 
As mentioned in the introduction, 
Theorem~\ref{theorem-lower-bound} gives a potential obstruction to 
the period-index conjecture, which predicts 
that $\ind(\alpha)$ divides $\per(\alpha)^{\dim(X)-1}$. 
This obstruction is trivial away from primes which are small relative to the dimension: 

\begin{lemma}
\label{lemma-PI-obstruction-vanish}
Let $X$ be a smooth projective complex variety with $\alpha \in \Br(X)$ a topologically trivial class. 
Let $b \in \rH^2(X, \bZ)$ be an integral $B$-field of degree $n = \per(\alpha)$ for $\alpha$, and let $B = b/n$ be the corresponding rational $B$-field. 
Then 
there exist Hodge classes $c_j$ as in the conclusion of Theorem~\ref{theorem-lower-bound} for $e$ a multiple of 
\begin{equation} 
\label{lcm-period}
\mathrm{lcm} \set{ \per(\alpha)^i i \st 1 \leq i \leq \dim(X)-1} . 
\end{equation}
In fact, if $\pi \colon P \to X$ is a Severi--Brauer variety of class $\alpha$ and relative dimension at least $e$, then the class  
\begin{equation*}
\gamma =h^e + \pi^*\left( \frac{e}{n} b \right) h^{e-1} + 
\pi^* \left( \frac{\binom{e}{2}}{n^2} b^2 \right)  + 
\cdots + \pi^* \left( \frac{ \binom{e}{\dim(X)-1}}{{n^{\dim(X)-1}}} b^{\dim(X)-1} \right) h^{e - \dim(X) + 1}. 
\end{equation*} 
satisfies the following: 
\begin{enumerate}
\item \label{gamma-integral}
$\gamma$ is an integral Hodge class, i.e. $ \gamma \in \rH^{e,e}(P, \bZ)$. 
\item \label{gamma-fibers}
$\gamma$ restricts to the fibers of $\pi \colon P \to X$ as the class of a codimension $e$ linear subspace. 
\item \label{gamma-algebraic}
$\per(\alpha)^e {\gamma}$ is algebraic, i.e. lies in the image of $\CH^e(P) \to \rH^{e,e}(P, \bZ)$. 
\end{enumerate} 
\end{lemma}

\begin{proof}
Note that $e \geq \dim(X)$ by our choice of $e$, so we only need to consider the polynomial $p_i^{B,e}$ from~\eqref{pBe} for $1 \leq i \leq \dim(X)$. Its leading term is 
\begin{equation*}
    p_i^{B,e}(0, \dots, 0) = 
    \binom{e}{i} \left(\frac{b}{n} \right)^i . 
\end{equation*}
For $1 \leq i \leq \dim(X) - 1$, if $e$ is a multiple of $n^{i}i$ then $n^i$ divides $\binom{e}{i}$, so for $e$ as in the lemma, 
the leading term above is integral for $1 \leq i \leq \dim(X)-1$. 
Therefore, taking  
$c_j = 0$ for $1 \leq j \leq \dim(X)-1$ and choosing $c_{\dim(X)} \in \rH^{2\dim(X)}(X, \bQ) = \bQ$ suitably so that $p_{\dim(X)}^{B,e}(0, \dots, 0, c_{\dim(X)}) = 0$ works.  
Explicitly, $c_{\dim(X)} = - \binom{e}{\dim(X)} \left(\frac{b}{n} \right)^{\dim(X)}$. 
This proves the first claim of the lemma. 

For the second, Theorem~\ref{theorem-divisibility} shows that $\gamma$ satisfies properties~\eqref{gamma-integral} and~\eqref{gamma-fibers}. 
Note that we can also write 
\begin{equation*}
\gamma = \left( h + \pi^* \left( \frac{b}{n} \right) \right)^{e} - \binom{e}{\dim(X)}\pi^*\left( \frac{b}{n} \right)^{\dim(X)}\left( h + \pi^* \left( \frac{b}{n} \right) \right)^{e-\dim(X)}
\end{equation*} 
Therefore, the algebraicity of $n^e\gamma$ follows from Lemma~\ref{lemma-cohomology-SB}\eqref{nhb-algebraic}. 
\end{proof}

\begin{remark}
\label{remark-obs-trivial} 
As $\ind(\alpha)$ and $\per(\alpha)$ have the same prime factors, $\ind(\alpha)$ divides~\eqref{lcm-period} if and only if it divides 
\begin{equation}
\label{obs-vanish}
\per(\alpha)^{\dim(X)-1} \cdot \prod_{p \mid \per(\alpha)} p^{\max \set{ v_p(i/\per(\alpha)^{\dim(X)-1-i}) \, \st \, 1\leq i \leq \dim(X)-1} }    , 
\end{equation}
where the product is over primes $p$ dividing $\per(\alpha)$ and $v_p$ denotes the $p$-adic valuation. Note that a prime $p$ can only appear nontrivially if it divides $(\dim(X)-1)!$.  The lemma says that Theorem~\ref{theorem-lower-bound} gives no obstruction to $\ind(\alpha)$ dividing~\eqref{obs-vanish}, which when $\per(\alpha)$ is prime to $(\dim(X)-1)!$ equals the power $\per(\alpha)^{\dim(X)-1}$ appearing in the period-index conjecture.  
\end{remark}

\begin{remark}
In the situation of Lemma~\ref{lemma-PI-obstruction-vanish}, in order to show that $\ind(\alpha)$ divides~\eqref{lcm-period}, it would suffice by Theorem~\ref{theorem-divisibility} to show that there exists a closed subvariety $Z \subset P$ of class~$\gamma$. 
An intermediate but still interesting question is whether the class $\gamma \in \rH^{e,e}(P, \bZ)$ is algebraic (as opposed to only the multiple $\per(\alpha)^e {\gamma}$ handled in the lemma). 
\end{remark}

Let us work out explicitly what the obstruction to the period-index conjecture says in the first interesting case, recovering \cite[Theorem 1.4]{hotchkiss-pi}: 

\begin{example}[Hotchkiss's threefold obstruction]
\label{example-dim3} 
Let $X$ be a smooth projective complex threefold, and let $\alpha \in \Br(X)$ be a topologically trivial class. 
We assume $n = \per(\alpha)$ is even, 
as by Remark~\ref{remark-obs-trivial} the obstruction from Theorem~\ref{theorem-lower-bound} to $\ind(\alpha)$ dividing $n^{2}$ is trivial otherwise. 
Let $B = b/n$ be a rational $B$-field for $\alpha$, where $b \in \rH^2(X, \bZ)$. 
If $\ind(\alpha)$ divides $n^2$, then by Theorem~\ref{theorem-lower-bound} there exist Hodge classes $c_1 \in \rH^{1,1}(X, \bQ)$, $c_2 \in \rH^{2,2}(X, \bQ)$, and $c_3 \in \rH^{3,3}(X, \bQ)$ such that 
\begin{align*}
    p_1^{B, n^2}(c_1) & = nb+ c_1 ,  \\ 
    p_2^{B, n^2}(c_1, c_2) & = \frac{n^2 -1}{2} b^2 + \frac{n^2-1}{n} b c_1 + c_2 , \\ 
    p_3^{B, n^2}(c_1, c_2, c_3) & = \binom{n^2}{3}B^3 + \binom{n^2 - 1}{2}B^2c_1 + (n^2 - 2)Bc_2 + c_3, 
\end{align*}
are integral. 
As all of $\rH^{6}(X, \bQ) = \bQ$ is Hodge, for any $c_1$ and $c_2$ we may choose $c_3$ so that the third expression vanishes, so this does not impose any constraints. 
The integrality of the first expression says that $c_1$ is integral, while multiplying the second by $n$ shows that 
\begin{equation*}
    \frac{n}{2}(n^2 -1)b^2 + (n^2-1)b c_1 + nc_2 \in n\rH^{4}(X, \bZ), 
\end{equation*}
and hence $nc_2$ is integral. 
In sum, if $\ind(\alpha)$ divides $n^2$ then there exists integral Hodge classes $c \in \rH^{1,1}(X, \bZ)$ and $d \in \rH^{2,2}(X, \bZ)$ such that 
\begin{equation*}
    \frac{n}{2} b^2 + bc + d \equiv 0 \pmod n. 
\end{equation*}
We do not know any reason for the existence of such $c$ and $d$ in general, but we also do not know an example where nonexistence can be shown. 
\end{example}

It is easy to similarly work out many Hodge-theoretic obstructions to the period-index conjecture in other dimensions. 
For instance, on fourfolds there is a possible obstruction when $\per(\alpha)$ is divisible by $3$, and so on. 

\subsection{Brauer classes of large index}
\label{section-large-index} 
It is interesting to find examples of Brauer classes whose index is the maximal one predicted by the period-index conjecture. 
Theorem~\ref{theorem-lower-bound} gives an effective method to do so, which we illustrate first in dimension $3$. 

\begin{example}[Sharpness criterion in dimension $3$]
\label{example-sharp-dim3} 
Let $X$ be a smooth projective complex threefold, 
let $b \in \rH^2(X, \bZ)$, let $n > 0$ be an integer, and let $\alpha = \exp(b/n) \in \Br(X)[n]$ be the corresponding topologically trivial class. 
If $\ind(\alpha)$ divides $n$, then using  Theorem~\ref{theorem-lower-bound} and arguing as in Example~\ref{example-kresch}, we find that there exist Hodge classes $c \in \rH^{1,1}(X, \bZ)$ and $d \in \rH^{2,2}(X, \bZ)$ such that 
\begin{equation} 
\label{equation-sharp} 
(n-1)b^2 \equiv 2bc + d \pmod{2n}.
\end{equation} 
If the period-index conjecture holds for $\alpha$ and $n$ is prime, then the failure of the existence of $c$ and $d$ implies that $\ind(\alpha) = n^2$, 
i.e. the divisibility bound predicted by the period-index conjecture is sharp for $\alpha$. 
Note that for $n = 2$ the above obstruction coincides with that of Example~\ref{example-kresch}. 
\end{example}

The criterion of Example~\ref{example-sharp-dim3}, as well as the corresponding sharpness criterion that one can derive in higher dimensions, can often by effectively checked for sufficiently generic varieties. 
We illustrate this on products of elliptic curves which are generic in the sense of Lemma~\ref{lemma-generic-product} below; 
in fact, we will produce Brauer classes of any index allowed by the period-index conjecture on such products. 

\begin{lemma}
\label{lemma-generic-product} 
Let $X = E_1 \times \cdots \times E_g$ be a product of $g$ pairwise non-isogenous complex elliptic curves.  
For each $1 \leq i \leq g$, let $x_i, y_i$ be a symplectic basis for $\rH^1(E_i, \bZ)$. 
Then: 
\begin{enumerate}
\item \label{H11-prod-Ei}
$\rH^{1,1}(X, \bZ) = \bigoplus_{i=1}^g \bZ(x_i \cup y_i)$, where 
we denote the pullbacks of $x_i$ and $y_i$ to $X$ by the same letters.  
\item \label{Hodge-prod-Ei}
The ring $\bigoplus_j \rH^{j,j}(X, \bQ)$ of rational Hodge classes is generated by $\rH^{1,1}(X, \bQ)$; 
explicitly,  
\begin{equation*}
\rH^{j,j}(X, \bQ) = \bigwedge^j \rH^{1,1}(X, \bQ) = 
\bigoplus_{i_1 < \cdots < i_j} \bQ(x_{i_1} \cup y_{i_1} \cup \cdots \cup x_{i_j} \cup y_{i_j}). 
\end{equation*} 
\end{enumerate}
In particular, the integral Hodge conjecture (Conjecture~\ref{conjecture-IHC} below) holds in all degrees for~$X$. 
\end{lemma} 

\begin{proof}
\eqref{H11-prod-Ei} holds as the $E_i$ are pairwise non-isogenous, while~\eqref{Hodge-prod-Ei} holds by \cite[Theorem~B.72]{survey-hodge}.
\end{proof} 

\begin{proposition}
\label{proposition-index-av}
Let $X = E_1 \times \cdots \times E_g$ be a product of $g \geq 2$ pairwise non-isogenous complex elliptic curves.   
Let $1 \leq t \leq g-1$ be an integer. 
Let $x_i, y_i$ be a symplectic basis for $\rH^1(E_i, \bZ)$, and let 
\begin{equation*}
    b = \sum_{i=1}^{t} x_i \cup y_{i+1} \in \rH^2(X, \bZ) . 
\end{equation*} 
For any integer $n > 0$, let $\alpha = \exp(b/n) \in \Br(X)[n]$ be the corresponding topologically trivial class. Then: 
\begin{enumerate}
\item \label{ind-divide-ng-1}
$\ind(\alpha)$ divides $n^{t}$. In fact, $\alpha$ is represented by an Azumaya algebra on $X$ of degree $n^{t}$. 
\item \label{ind-not-divide-ng-2}
If $n$ does not divide $(t-1)!$, then $\ind(\alpha)$ does not divide $n^{t-1}$. 
\end{enumerate}  
In particular, if $n$ is a prime power which does not divide $(t-1)!$, then $\ind(\alpha) = n^{t}$. 
\end{proposition} 

\begin{proof}
\eqref{ind-divide-ng-1} follows from a general observation about Brauer classes obtained as cup products. 
Let $x, y \in \rH^1(X, \bZ)$ and let $\beta = \exp(x \cup y/n) \in \Br(X)[n]$ be the corresponding $n$-torsion class. 
If $x_n , y_n \in \rH^1(X, \bmu_n)$ are the images of $x$ and $y$, then 
$\beta$ can also be described as the image in $\Br(X)[n]$ of the cup product $x_n \cup y_n \in \rH^2(X, \bmu_n)$. 
It follows that $\beta$ is represented by a cyclic Azumaya algebra on $X$ of degree $n$. 
Therefore, $\alpha$ can be represented by a product of $t$ such Azumaya algebras, which has degree $n^{t}$. 

For~\eqref{ind-not-divide-ng-2}, we use the obstruction of Theorem~\ref{theorem-lower-bound} for $B = b/n$, $e = n^{t-1}$, and $i = t$ (note that $t \leq n^{t-1}$ for $n > 1$). 
Namely, writing $p = p_{t}^{b/n, n^{t-1}}$ for simplicity, then in order for $\ind(\alpha)$ to divide $n^{t-1}$ there must exist classes 
$c_j \in \rH^{j,j}(X, \bQ)$ for $1 \leq j \leq t$ such that $p(c_1, \dots, c_{t})$ is integral. 
By definition, we have 
\begin{equation}
\label{pci} 
p(c_1, \dots, c_{t}) = 
\binom{n^{t-1}}{t} \frac{b^{t}}{n^{t}} + \binom{n^{t-1}-1}{t-1} \frac{b^{t-1}}{n^{t-1}} c_1 + 
\binom{n^{t-1}-2}{t-2} \frac{b^{t-2}}{n^{t-2}} c_2 + \cdots + c_{t}. 
\end{equation} 
The leading term can be rewritten as 
\begin{equation}
\label{equation-bg-1}
\frac{1}{n}(n^{t-1}-1)(n^{t-1}-2) \cdots (n^{t-1}-(t-1)) \omega,  
\end{equation} 
where $\omega =  x_1 \cup y_2 \cup x_2 \cup y_3 \cup \cdots \cup x_{t} \cup y_{t+1}$. 
Note that if (and only if) $n$ does not divide $(t-1)!$, then~\eqref{equation-bg-1} is not integral. 
We claim that for $1 \leq j \leq t$ any element in $b^{t-j} \rH^{j,j}(X, \bQ)$, when expressed in terms of the basis for $\rH^{2t}(X, \bZ)$ given by monomials in the $x_i$ and $y_i$, has no $\omega$ term. 
Together with the previous observation about the non-integrality of~\eqref{equation-bg-1}, this 
implies that there do not exist $c_j \in \rH^{j,j}(X, \bQ)$ making~\eqref{pci} integral when $n$ does not divide $(t-1)!$. 

It remains to prove the above claim. 
Note that $b^{t-j}$ is contained in the span of monomials of the form 
\begin{equation*}
x_{k_1} \cup y_{k_1 +1} \cup x_{k_2} \cup y_{k_2+1} \cup \cdots \cup x_{k_{t-j}} \cup y_{k_{t-j}+1}, 
\quad 1 \leq k_1 < \cdots <k_{t-j} \leq t. 
\end{equation*} 
Since $t-j < t$, any such monomial must not include the factor $x_k \cup y_{k+1}$ for some $1 \leq k \leq t$. 
Using Lemma~\ref{lemma-generic-product}\eqref{Hodge-prod-Ei}, 
it is then easy to see that the product of such a monomial with any element of $\rH^{j,j}(X, \bQ)$ does not contain an $\omega$ term. 
\end{proof} 

\begin{remark}
By a similar but easier argument, one can also prove a version of Proposition~\ref{proposition-index-av} with $X$ replaced by a very general abelian variety. This would also suffice for the construction of counterexamples to the integral Hodge conjecture as Severi--Brauer varieties over $X$ in \S\ref{section-IHC} below. 
The case when $X$ is a product of elliptic curves as above is especially nice because 
then $X$ itself satisfies the integral Hodge conjecture (Lemma~\ref{lemma-generic-product}), and because 
Proposition~\ref{proposition-index-av} transports well to $\overline{\bF}_p$, as we explain in Proposition~\ref{proposition-index-av-Fp}. 
\end{remark}

\subsection{Another divisibility obstruction} 
\label{section-index-via-0-cycles}
Let $\alpha \in \Br(X)$ where $X$ is a smooth projective variety. 
The paradigm behind Theorem~\ref{theorem-divisibility}, as well as the definition of $\ind_{\Hdg}(\alpha)$, is to take a characterization of the index of $\alpha \vert_{k(X)}$ over the generic point, translate it to a global condition on $X$, 
and then consider the Hodge-theoretic avatar of the global condition. 
Namely, the description of $\ind(\alpha)$ in terms of twisted sheaves from Lemma~\ref{lemma-index-unramified} leads to $\ind_{\Hdg}(\alpha)$, while the description in terms of twisted linear subspaces from Lemma~\ref{lemma-index-SB} leads to Theorem~\ref{theorem-divisibility}. 
It is natural to also consider the description of $\ind(\alpha)$ in terms of $0$-cycles on a Severi--Brauer variety from Lemma~\ref{lemma-index-SB}. 
As we will see, this point of view tends to lead to weaker divisibility obstructions than Theorem~\ref{theorem-divisibility}, but as we explain in \S\ref{section-IHC}, this discrepancy is interesting because it leads to non-algebraic integral Hodge classes. 

To formulate the obstruction we get in this way, 
for $B \in \rH^2(X, \bQ)$ and positive integers $r$ and $e$, we define for $1 \leq i \leq \min \set{ r, \dim(X) }$ the polynomial 
\begin{equation}
\label{qBre}
    q_i^{B,r,e}(x_1, \dots, x_i) =
    e\binom{r}{i}B^i + 
    \sum_{j=1}^{i} \binom{r-j}{i-j} B^{i-j} x_j  \in \rH^*(X, \bQ)[x_1, \dots, x_i] , 
\end{equation}
similar to $p_i^{B,e}$ from~\eqref{pBe}. 
Then arguing exactly as in the proof of Theorem~\ref{theorem-divisibility} gives the following result, already implicit in \cite[\S6.1]{hotchkiss-pi}. 

\begin{proposition}
\label{proposition-0-cycle}
Let $X$ be a smooth projective variety with $\alpha \in \Br(X)$. 
Let $P \to X$ be a Severi--Brauer variety of class $\alpha$ and relative dimension $r$. 
Let $e$ be an integer. 
\begin{enumerate}
\item \label{0-cycle-1} 
$\ind(\alpha)$ divides $e$ if and only if there exists a cycle $Z \in \CH^r(P)$ 
whose generic fiber $Z_{\eta} \in \CH_0(P_{\eta})$ is a $0$-cycle of degree $e$. 
In particular, if $\ind(\alpha)$ divides $e$, there exists an integral Hodge class $\delta \in \rH^{r,r}(P, \bZ)$ whose restriction to fibers has degree $e$. 
\item 
\label{0-cycle-top-triv}
If $\alpha$ is topologically trivial and $B \in \rH^2(X, \bQ)$ is a rational $B$-field for $\alpha$, 
then there exists a class $\delta \in \rH^{r,r}(P, \bZ)$ satisfying the above condition if and only if there exist Hodge classes 
\begin{equation*}
c_j \in \rH^{j,j}(X, \bQ), \quad 1 \leq j \leq \min \set{ r, \dim(X) }, 
\end{equation*} 
such that  
$q_i^{B,r,e}(c_1, \dots, c_i) \in \rH^{2i}(X, \bQ)$ is integral for all $1 \leq i \leq \min \set{ r, \dim(X) }$. 
Explicitly, $\delta$ and the $c_j$ are related by  
\begin{align*}
    \delta & = e(h+\pi^*(B))^r + \pi^*(c_1)(h + \pi^*(B))^{r-1} + \cdots + \pi^*(c_m)(h+\pi^*(B))^{r-m} \\ 
    & = eh^r + \pi^*(q_1^{B,r,e}(c_1))h^{r-1} + 
    \pi^*(q_2^{B,r,e}(c_1,c_2))h^{r-2} + \cdots + \pi^*(q_{m}^{B,r,e}(c_1, \dots, c_m)) h^{r - m} , 
\end{align*} 
where $m = \min \set{ r, \dim(X) }$. 
\end{enumerate}
\end{proposition}

\begin{remark}
\label{remark-zero-cyle-obstruction}
Interestingly, it is easy to see that Proposition~\ref{proposition-0-cycle} does \emph{not} lead to a potential obstruction to the period-index conjecture similar to the one discussed in \S\ref{section-obstruction-PI}. 
Namely, if 
$b \in \rH^2(X, \bZ)$ is an integral $B$-field of degree $n =  \per(\alpha)$ for $\alpha$ and 
$B = b/n$ is the corresponding rational $B$-field, then Hodge classes $c_j$ as in Proposition~\ref{proposition-0-cycle}\eqref{0-cycle-top-triv} for $e = n^{\dim(X)-1}$ exist: take $c_j = 0$ for $j < \dim(X)$ and, if $\dim(X) \leq r$,  
$c_{\dim(X)} \in \rH^{2\dim(X)}(X, \bQ) = \bQ$ suitably so that $q_{\dim(X)}^{B,r,e}(0,\dots, 0, c_{\dim(X)}) = 0$; explicitly, $c_{\dim(X)} = - n^{\dim(X)-1} \binom{r}{\dim(X)} B^{\dim(X)}$. 
\end{remark} 

Similar to \S\ref{section-large-index}, we can attempt to use Proposition~\ref{proposition-0-cycle} 
to construct Brauer classes of large index. 
However, this seems to generally give worse bounds than those obtained from Theorem~\ref{theorem-lower-bound}.  
We illustrate this with two examples below. 

\begin{example}
Let us work out the analog of Example~\ref{example-sharp-dim3}. 
Let $X$ be a smooth projective complex threefold, 
let $\alpha \in \Br(X)$ be a topologically trivial class of period $n$, and 
write $\alpha = \exp(b/n)$ for some $b \in \rH^2(X, \bZ)$. 
Let $P \to X$ be a Sever--Brauer variety of class $\alpha$ and relative dimension $r$. 
Using Proposition~\ref{proposition-0-cycle}, we find that if $\ind(\alpha)$ divides $n$ then 
there exists a Hodge class $\delta \in \rH^{r,r}(P, \bZ)$ of degree $n$ on fibers, 
and such a $\delta$ exists if and only if there exist Hodge classes $c \in \rH^{1,1}(X, \bZ)$ and $d \in \rH^{2,2}(X, \bZ)$ such that 
\begin{equation*}
b^2 \equiv 2bc + d \pmod{n} . 
\end{equation*} 
This is weaker than the criterion of Example~\ref{example-sharp-dim3}. 
\end{example} 

\begin{lemma}
\label{lemma-weaker-bound-AV}
Let $X$ be a complex abelian variety of dimension $g$. 
Let $\alpha \in \Br(X)[n]$ be an $n$-torsion Brauer class, and let 
$\pi \colon P \to X$ be a Severi--Brauer variety of class $\alpha$ and relative dimension $r$. 
If $n$ divides $(g-1)!$, then there exists an integral Hodge class $\delta \in \rH^{r,r}(P, \bZ)$ whose 
restriction to fibers has degree $n^{g-2}$. 
More precisely, if $b \in \rH^{2}(X, \bZ)$ is an integral $B$-field of degree $n$ for $\alpha$, then the class 
\begin{equation}
\label{delta-av}
\delta = 
\begin{cases}
\sum_{i=0}^{g-1} \pi^* \left( n^{g-2-i}r(r-1)\cdots(r-(i-1)) \frac{b^i}{i!} \right) h^{r-i} & \text{if }  g \leq r \\
\sum_{i=0}^{r} \pi^* \left( n^{g-2-i}r(r-1)\cdots(r-(i-1)) \frac{b^i}{i!} \right) h^{r-i} & \text{if }  r < g. 
\end{cases} 
\end{equation} 
satisfies the following:
\begin{enumerate}
\item $\delta$ is an integral Hodge class, i.e. $\delta \in \rH^{r,r}(P, \bZ)$. 
\item $\delta$ restricts to the fibers of $\pi \colon P \to X$ as a degree $n^{g-2}$ class. 
\item $n^{r-g-2} \delta$ is algebraic if $r>g-2$, and $\delta$ is algebraic if $r \leq g-2$. 
\end{enumerate} 
\end{lemma} 

\begin{remark} 
This should be contrasted with Proposition~\ref{proposition-index-av}, 
where it is shown that if $n$ does not divide $(g-2)!$ (but possibly $(g-1)!$), then 
Theorem~\ref{theorem-lower-bound} obstructs $\ind(\alpha)$ dividing $n^{g-2}$ for certain choices of abelian varieties $X$ and $\alpha$. 
\end{remark} 

\begin{remark}
As the proof below shows, the hypothesis that $n$ divides $(g-1)!$ is not needed when $r \leq g-2$; 
moreover, in this case by Proposition~\ref{proposition-0-cycle}\eqref{0-cycle-1}, $\ind(\alpha)$ divides $n^{g-2}$. 
The more interesting case is when $r$ is large relative to $g$ and $\ind(\alpha)$ does not divide $n^{g-2}$, 
as this will be needed for our counterexamples to the integral Hodge conjecture constructed in \S\ref{section-IHC}. 
\end{remark} 

\begin{proof}
Note that $\alpha$ is topologically trivial because $\rH^3(X, \bZ)$ is torsion-free. 
If we set $B = b/n$ and $e = n^{g-2}$, then 
by Proposition~\ref{proposition-0-cycle}\eqref{0-cycle-top-triv} 
it suffices to check that there exist Hodge classes $c_j \in \rH^{j,j}(X, \bQ)$ such that 
$q_i^{B,r,e}(c_1, \dots, c_i) \in \rH^{2i}(X, \bQ)$ is integral for $1 \leq i \leq \min \set{r, g}$. 
The leading term of $q_i^{B,r,e}(c_1, \dots, c_i)$ can be written as   
\begin{equation}
\label{qi-leading-term}
n^{g-2-i} r(r-1) \cdots (r-(i-1)) \frac{b^i}{i!}. 
\end{equation} 
As $X$ is an abelian variety, $\frac{b^i}{i!}$ is an integral class. 
The coefficient $n^{g-2-i} r(r-1) \cdots (r-(i-1))$ is obviously integral when $1 \leq i \leq g-2$. 
For $i = g-1$, it is integral when $n$ divides $r(r-1)\cdots(r-(g-2))$, 
or equivalently when $n$ divides $(g-1)!$ as $r \equiv -1 \pmod{n}$ (see Lemma~\ref{lemma-index-SB}). 
Altogether, this shows that if $n$ divides $(g-1)!$, then we may take $c_j = 0$ for $1 \leq j \leq g-1$ 
to make $q_i^{B,r,e}(c_1, \dots, c_i) \in \rH^{2i}(X, \bQ)$ integral for $1 \leq i \leq g-1$. 
We can then take $c_g \in \rH^{2g}(X, \bQ) = \bQ$ suitably so that $q_g^{B,r,e}(0, \dots, 0, c_g) = 0$. 

The explicit formula for $\delta$ and the fact that it satisfies the 
desired properties follows from the description of $\delta$ in terms of the $c_j$ in Proposition~\ref{proposition-0-cycle}\eqref{0-cycle-top-triv} 
(cf. the proof of Lemma~\ref{lemma-PI-obstruction-vanish}). 
\end{proof} 

\subsection{Counterexamples to the integral Hodge conjecture} 
\label{section-IHC} 
Our main interest in Proposition~\ref{proposition-0-cycle} is that it can be used to construct counterexamples to the integral Hodge conjecture, which states the following. 

\begin{conjecture}[Integral Hodge conjecture in codimension $c$]
\label{conjecture-IHC} 
Let $X$ be a smooth projective complex variety. 
Then the cycle class map $\CH^c(X) \to \rH^{c,c}(X, \bZ)$ is surjective. 
\end{conjecture} 

Although called a conjecture, it is well-known to be false in general. 
Atiyah and Hirzebruch constructed the first counterexamples \cite{AH}, but since then there have been a plethora of others. 
To understand how far the conjecture is from failing, one would like counterexamples on varieties that are as ``nice'' as possible. 
Hotchkiss \cite[\S6]{hotchkiss-pi} introduced a method for producing particularly simple counterexamples on Severi--Brauer varieties. 
To formalize this, 
it is convenient to introduce some notation. 

\begin{definition}
\label{definition-ind-HP}
Let $X$ be a smooth projective complex variety. 
Let $\alpha \in \Br(X)$ and let $P \to X$ be a Severi--Brauer variety of class $\alpha$ and relative dimension $r$. 
The \emph{Hodge-theoretic index of $\alpha$ with respect to $P$} is 
\begin{equation*}
\ind_{\Hdg}^P(\alpha) \coloneqq \min  \set{ e \st \exists ~ \delta \in \rH^{r,r}(P, \bZ) \text{ of degree $e > 0$ on fibers of $P \to X$} }. 
\end{equation*} 
\end{definition} 

Note that it follows from Proposition~\ref{proposition-0-cycle} that $\ind_{\Hdg}^P(\alpha)$ divides $\ind(\alpha)$. 
\begin{lemma}
Let $X$ be a smooth projective complex variety. 
Let $\alpha \in \Br(X)$ and let $P \to X$ be a Severi--Brauer variety of class $\alpha$ and relative dimension $r$. 
If $\ind_{\Hdg}^P(\alpha) < \ind(\alpha)$, 
then the integral Hodge conjecture in codimension $r$ fails on $P$. 
More precisely, if $\delta \in \rH^{r,r}(P, \bZ)$ is a class whose degree on fibers is not divisible by $\ind(\alpha)$, then 
$\delta$ is not in the image of the cycle class map $\CH^r(P) \to  \rH^{r,r}(P, \bZ)$. 
\end{lemma} 

\begin{proof}
Indeed, if $\delta$ were the class of a cycle $Z \in \CH^r(P)$, then $\ind(\alpha)$ would divide the fibral degree of $\delta$ by  Proposition~\ref{proposition-0-cycle}\eqref{0-cycle-1}. 
\end{proof} 

Together with Lemma~\ref{lemma-weaker-bound-AV}, we obtain the following criterion for Severi--Brauer varieties 
over an abelian variety to violate the integral Hodge conjecture, generalizing the $g = 3$ case treated in \cite[Theorem 6.1]{hotchkiss-pi}. 
\begin{corollary}
\label{corollary-ihc-fail-av}
Let $X$ be a complex abelian variety of dimension $g$. 
Let $\alpha \in \Br(X)[n]$ be an $n$-torsion Brauer class such that 
$n$ divides $(g-1)!$ and
$\ind(\alpha)$ does not divide $n^{g-2}$. 
Then for any Severi--Brauer variety $P \to X$ of relative dimension $r$, 
the integral Hodge conjecture in codimension $r$ fails on $P$. 
More precisely, the class $\delta \in \rH^{r,r}(P, \bZ)$ from~\eqref{delta-av} is not algebraic. 
\end{corollary} 

As explained in \S\ref{section-index-via-0-cycles}, $\ind_{\Hdg}^P(\alpha)$ can often be less than lower bounds on $\ind(\alpha)$ derived from Theorem~\ref{theorem-lower-bound}. Any such discrepancy therefore leads to counterexamples to the integral Hodge conjecture, 
as we illustrate by the following result. 

\begin{theorem}
\label{theorem-IHC} 
Let $X = E_1 \times \cdots \times E_g$ be a product of $g \geq 3$ pairwise non-isogenous complex elliptic curves. 
Let $n > 0$ be an integer such that $n$ divides $(g-1)!$ but not $(g-2)!$, 
and let $\alpha \in \Br(X)[n]$ be the class from Proposition~\ref{proposition-index-av} for $t = g-1$. 
Let $P \to X$ be a Severi--Brauer variety of class $\alpha$ and relative dimension $n^{g-1}-1$, which 
exists by Proposition~\ref{proposition-index-av}\eqref{ind-divide-ng-1}. 
Then the integral Hodge conjecture in codimension $n^{g-1}-1$ fails on $P$. 
More precisely, the class $\delta \in \rH^{r,r}(P, \bZ)$ with $r = n^{g-1}-1$ from~\eqref{delta-av} is not algebraic. 
\end{theorem} 

\begin{proof}
By Proposition~\ref{proposition-index-av}\eqref{ind-not-divide-ng-2} $\ind(\alpha)$ does not divide $n^{g-2}$, so Corollary~\ref{corollary-ihc-fail-av} applies. 
\end{proof} 

For $g = 3$ and $n=2$, examples as in Theorem~\ref{theorem-IHC} were first constructed by Hotchkiss \cite[\S6.1]{hotchkiss-pi}, 
using a result of Gabber \cite[Appendice]{CT-examples} instead of Proposition~\ref{proposition-index-av}.  
Proposition~\ref{proposition-index-av} allows this argument to be extended to higher dimensions. 
Theorem~\ref{theorem-IHC} is also the model for Theorem~\ref{theorem-ITC}, which is its counterpart over $\overline{\bF}_p$. 
In that setting, the higher-dimensional case is crucial for making $\ell$-adic examples for every~$\ell$. 


\section{Lower bounds on the index over other fields} 
\label{section-lower-bound-k} 

In this section, we explain how to extend the results of \S\ref{section-lower-bound} to more general base fields. 
The main task is to define a notion of a ``topologically trivial'' Brauer class over an arbitrary algebraically closed base field, and describe the cohomology of Severi--Brauer varieties of such a class. 

We will focus on Brauer classes which are $\ell$-power torsion for a prime $\ell$ which is distinct from the characteristic $p$ of the base field. 
There is little loss of generality in doing so, as the study of the index of Brauer classes reduces to the case of prime-power torsion classes (see Lemma~\ref{lemma-index-bound-divisible} and its proof). 
Though we do not do so here, it should also be possible to prove analogs of our results for $\ell = p$ (see Remark~\ref{remark-crystalline}). 

\subsection{$\ell$-adic topologically trivial Brauer classes}
In Lemma~\ref{lemma-top-triv} we gave several equivalent characterizations of a topologically trivial Brauer class on a complex variety. 
The third has a natural analog in the $\ell$-adic setting: 

\begin{definition}
Let $X$ be a variety over an algebraically closed field $k$, and let $\ell \neq \characteristic(k)$ be a prime. 
An element $\alpha \in \Br(X)[\ell^{\infty}]$ is called an \emph{$\ell$-adic topologically trivial Brauer class} if it lies in the image of the natural composition 
\begin{equation}
\label{map-b-l}
    \rH^{2}_{\et}(X, \bZ_{\ell}(1)) \to \rH^2_{\et}(X, \bmu_{\ell^n}) \twoheadrightarrow \Br(X)[\ell^n]
\end{equation}
for some integer $n \geq 0$. 
If $b \in \rH^{2}_{\et}(X, \bZ_{\ell}(1))$ maps to $\alpha$ under the above composition, we say that $b$ is an \emph{integral $\ell$-adic $B$-field of degree $\ell^n$} for $\alpha$. 
\end{definition}

This notion was introduced (using slightly different terminology) in \cite[\S{3.2}]{LMS}. 
The subset $\Br_{\ell}^{\toptriv}(X) \subset \Br(X)[\ell^{\infty}]$ of $\ell$-adic topologically trivial classes is easily seen to be a subgroup \cite[Lemma 3.2.5]{LMS}. 

\begin{remark}
\label{remark-delta} 
For any integer $n \geq 0$, there is a surjective homomorphism 
\begin{equation*}
    \delta_{n} \colon \Br(X)[\ell^n] \to \rH^3_{\et}(X, \bZ_{\ell}(1))[\ell^n] 
\end{equation*}
defined by the commutative diagram 
\begin{equation*}
\xymatrix{
  0 \ar[r]  & \Pic(X)/\ell^n \ar[r] \ar[d]_{c_1} &  \rH^2_{\et}(X, \bmu_{\ell^n}) \ar[r] \ar@{=}[d] & \Br(X)[\ell^n] \ar[r] \ar[d]_{\delta_n} & 0 \\ 
    0\ar[r] & \rH^2(X, \bZ_{\ell}(1))/\ell^n \ar[r] &  \rH^2_{\et}(X, \bmu_{\ell^n}) \ar[r] & \rH^3_{\et}(X, \bZ_{\ell}(1))[\ell^n] \ar[r] & 0 
}
\end{equation*}
where the top row is given by the cohomology of the Kummer sequence and the second by the cohomology of $0 \to \bZ_{\ell}(1) \to \bZ_{\ell}(1) \to \bmu_{\ell^n} \to 0$. 
It is straighforward to check that the maps $\delta_n$ are compatible as $n$ varies, and hence give rise to a surjective homomorphism 
\begin{equation*}
    \delta \colon \Br(X)[\ell^{\infty}] \to \rH_{\et}^{3}(X, \bZ_{\ell}(1))_{\tors}. 
\end{equation*}
Moreover, it follows from the definitions that $\ker(\delta) = \Br_{\ell}^{\toptriv}(X)$ is precisely the subgroup of $\ell$-adic topologically trivial classes. 
This should be regarded as the $\ell$-adic analog of Lemma~\ref{lemma-top-triv}\eqref{top-triv-3}. 
\end{remark}

\begin{remark}
\label{ker-reduction} 
The above analysis also shows that there is an exact sequence 
\begin{equation*}
    \Pic(X) \oplus \ell^n \rH^2_{\et}(X, \bZ_{\ell}(1)) \to \rH^2_{\et}(X, \bZ_{\ell}(1)) \to \Br_{\ell}^{\toptriv}(X)[\ell^n] \to 0 
\end{equation*}
where the first map is given by $c_1$ on the first summand and the inclusion on the second, and the second map is given by~\eqref{map-b-l}.
\end{remark}

\begin{remark}
There is a surjective homomorphism  
\begin{equation*}
    \varepsilon \colon \rH^{2}_{\et}(X, \bQ_{\ell}(1)) \to \Br_{\ell}^{\mathrm{tt}}(X)
\end{equation*}
given as follows: for $B \in \rH^{2}_{\et}(X, \bQ_{\ell}(1))$, choose $n \geq 0$ so that $\ell^n B \in \rH^2(X, \bZ_{\ell}(1))$, and define $\varepsilon(B)$ to be the image of $\ell^nB$ under the map~\eqref{map-b-l}. The map $\varepsilon$ is easily seen to be well-defined, and by definition it is surjective. It should be regarded as the $\ell$-adic analog of the exponential map~\eqref{exp}. 
Given $\alpha \in \Br_{\ell}^{\mathrm{tt}}(X)$, we call an element $B \in \rH^{2}_{\et}(X, \bQ_{\ell}(1))$ such that $\varepsilon(B) = \alpha$ a \emph{rational $\ell$-adic $B$-field} for $\alpha$. 
\end{remark}

Having formulated the notion of $\ell$-adic topological triviality, we now turn to the $\ell$-adic analog of Lemma~\ref{lemma-cohomology-SB}. 

\begin{lemma}
\label{lemma-cohomology-SB-k} 
Let $X$ be a smooth proper variety over an algebraically closed field $k$, and let $\ell \neq \characteristic(k)$ be a prime.
Let $\alpha \in \Br(X)$ be an $\ell$-adic topologically trivial class. 
Let $\pi \colon P \to X$ be a Severi--Brauer variety of class $\alpha$ and relative dimension $r$. 
\begin{enumerate}
    \item \label{k-top-h} 
    There exists an element $h \in \rH_{\et}^2(P, \bZ_{\ell}(1))$, unique modulo classes pulled back from $X$, whose restriction $h_{\bar{x}}$ to any fiber $P_{\bar{x}} \cong \bP_{\kappa(\bar{x})}^r$ over a geometric point $\bar{x}$ of $X$ is the hyperplane class. 
    \item \label{integral-coh-P} 
    For any integer $d$, there is an isomorphism  
\begin{equation}
\label{equation-H-P-k}
    \sum_i \pi^*(-) \cup h^j \colon \bigoplus_{j = 0}^r \rH_{\et}^{d - 2j}(X, \bZ_{\ell}(-j))  \xrightarrow{\, \sim \,} \rH^d_{\et}(P, \bZ_{\ell}). 
\end{equation} 
    \item \label{h-B-algebraic}
    If $b \in \rH_{\et}^{2}(X, \bZ_{\ell}(1))$ is an integral $\ell$-adic $B$-field of degree $\ell^n$ for $\alpha$, 
    then there exists $h$ as in~\eqref{k-top-h} such that $\ell^nh + \pi^*b \in \rH^{2}_{\et}(P, \bZ_{\ell}(1))$ is algebraic. 
    
 \item \label{coh-P-ladic} For $b$ and $h$ as in~\eqref{h-B-algebraic} and $B = b/\ell^n \in \rH_{\et}^2(X, \bQ_{\ell}(1))$ the corresponding rational $\ell$-adic $B$-field,  
 there is an isomorphism 
\begin{equation}
\label{equation-coh-P}
\sum_j \pi^*(-) \cup (h+\pi^*B)^{j} \colon 
\bigoplus_{j=0}^r \rH^{d-2j}_{\et}(X, \bQ_{\ell}(-j)) \xrightarrow{\, \sim \,} \rH_{\et}^d(P, \bQ_{\ell}) . 
\end{equation}
\end{enumerate}
\end{lemma}

\begin{proof}
Note that $\rR^0\pi_* \bZ_{\ell}= \bZ_{\ell}$, $\rR^1\pi_* \bZ_{\ell} = 0$, and $\rR^{2}\pi_* \bZ_{\ell} = \bZ_{\ell}$, so the Leray spectral sequence gives an exact sequence 
\begin{equation}
\label{leray-es} 
    0 \to \rH^2_{\et}(X, \bZ_{\ell}(1)) 
    \xrightarrow{ \pi^* } \rH^2_{\et}(P, \bZ_{\ell}(1)) 
    \to \rH^0_{\et}(X, \rR^2\pi_*\bZ_{\ell}(1)) \xrightarrow{ \, d_3 \, } \rH^3_{\et}(X, \bZ_{\ell}(1))
\end{equation}
where the second map is given by restriction and $d_3$ is the $E_3$-differential in the spectral sequence. 
The generator $\bar{h} \in \rH_{\et}^0(X, \rR^{2}\pi_* \bZ_{\ell}(1)) \cong \bZ_{\ell}(1)$ restricts over geometric fibers of $\pi$ to the hyperplane class, so for~\eqref{k-top-h} we need to show that $\bar{h}$ lifts to $\rH^2_{\et}(P, \bZ_{\ell}(1))$. 
But $d_3(\bar{h}) \in \rH^3_{\et}(X, \bZ_{\ell}(1))$ can be identified with $\delta(\alpha)$, where $\delta \colon \Br(X)[\ell^{\infty}] \to \rH_{\et}^{3}(X, \bZ_{\ell}(1))_{\tors}$ is the map constructed in Remark~\ref{remark-delta}, so the claim follows from the $\ell$-adic topological triviality of $\alpha$. 
(We omit the proof of this description of $d_{3}(\bar{h})$, but our proof of~\eqref{h-B-algebraic} below gives an independent argument for the existence of $h$.) 

The classes $1, h, h^2, \dots, h^{r}$ restrict on geometric fibers $P_{\bar{x}}$ to a basis of $\rH^*_{\et}(P_{\bar{x}}, \bZ_{\ell})$. 
This implies~\eqref{integral-coh-P} by Leray--Hirsch. 

Now we prove~\eqref{h-B-algebraic}. 
Consider the $\pi^*\alpha$-twisted line bundle $\cO_P^{\alpha}(1)$ on $P$ from Remark~\ref{remark-Oalpha1}. 
Then $\cO_P^{\alpha}(1)^{\otimes \ell^n}$ is an ordinary line bundle; let $c \in \rH^2_{\et}(P, \bZ_{\ell}(1))$ be its first Chern class. 
Denoting by $b_n$ the image of $b$ in $\rH^2_{\et}(X, \bmu_{\ell^n})$ and by $c_n$ the image of $c$ in $\rH^2_{\et}(P, \bmu_{\ell^n})$, one may check that $c_n = \pi^*b_n$ (for instance, use the description of $c_1 \colon \Pic(P) \to \rH^2_{\et}(P, \bmu_{\ell^n})$ as the boundary map in the Kummer sequence and a \v{C}ech cohomology computation). 
Therefore $c - \pi^*b \in \rH^2_{\et}(P, \bZ_{\ell}(1))$ is divisible by $\ell^n$, say equal to $\ell^n h$ for some $h \in \rH^{2}_{\et}(P, \bZ_{\ell}(1))$. 
As $c$ restricts to $\ell^n$ times the hyperplane class on fibers, it follows that $h$ restricts to the hyperplane class on fibers. 
Altogether, we have shown that $\ell^n h + \pi^*b = c$, 
with $h$ as in~\eqref{k-top-h} and $c$ the first Chern class of a line bundle, hence algebraic

Finally, \eqref{coh-P-ladic} follows formally from~\eqref{k-top-h}. 
\end{proof}

\begin{remark}[Tate classes]
\label{remark-coh-P}
Assume $k$ is the algebraic closure of a finitely generated field. 
For any model $X_0$ of $X$ defined over a field $k_0 \subset k$, we obtain a $\Gal(k/k_0)$-action on $\rH^{2i}_{\et}(X, \bQ_{\ell}(i))$. 
A \emph{Tate class} $c \in \rH^{2i}_{\et}(X, \bQ_{\ell}(i))$ is a cohomology class fixed by the $\Gal(k/k_0)$-action corresponding to some model $X_0$ defined over a finitely generated field $k_0$; similarly, an \emph{integral Tate class} is defined by replacing $\bQ_{\ell}$ with $\bZ_{\ell}$. 
The image of the cycle class map $\CH^i(X) \otimes \bQ_{\ell} \to \rH^{2i}_{\et}(X, \bQ_{\ell}(i))$ is contained in the subgroup $\rH^{2i}_{\et}(X, \bQ_{\ell}(i))^{\Tate}$ of Tate classes; similarly, 
the image of $\CH^i(X) \otimes \bZ_{\ell} \to \rH^{2i}_{\et}(X, \bZ_{\ell}(i))$ is contained in the subgroup 
$\rH^{2i}_{\et}(X, \bZ_{\ell}(i))^{\Tate}$ of integral Tate classes. 

In the situation of Lemma~\ref{lemma-cohomology-SB-k}, note that
if $k_0 \subset k$ is a subfield over which there exists a model $\pi_0 \colon P_0 \to X_0$ of $\pi \colon P \to X$ and a cycle in $\CH^1(X_0) \otimes \bQ_{\ell}$ of class $h + \pi^*B$, then~\eqref{equation-coh-P} is an isomorphism of $\Gal(k/k_0)$-modules. 
It follows that the isomorphism~\eqref{equation-coh-P} for $d = 2i$ induces an isomorphism 
\begin{equation*}
\bigoplus_{j=0}^r \rH^{2i-2j}_{\et}(X, \bQ_{\ell}(i-j))^{\Tate} \xrightarrow{\, \sim \,} \rH_{\et}^{2i}(P, \bQ_{\ell}(i))^{\Tate}  
\end{equation*}
on the subgroups of Tate classes. 
\end{remark}

\begin{remark}
\label{remark-crystalline}
When $\alpha \in \Br(X)[p^{\infty}]$ where $p = \characteristic(k)$, a different approach in terms of crystalline cohomology is needed to define notions of ``topological triviality'' or ``B-fields''. 
Crystalline B-fields are discussed in \cite[\S4]{derived-equivalences-bragg}, \cite[\S4.4]{twistor-bl}, and \cite[\S2]{bragg-yang} in the case when $X$ is a K3 surface, but the assumption that $X$ is a K3 surface is not essential for the basic theory.  
Using these ideas, it should be possible to formulate and prove $p$-adic versions of the results in this section. 
\end{remark}

\subsection{Divisibility obstructions}
\label{section-divisibility-obstructions-char-p}
Now we formulate the $\ell$-adic analog of Theorem~\ref{theorem-divisibility}. 
For $B \in \rH_{\et}^2(X, \bQ_{\ell}(1))$, 
we write $p_{i}^{B,e}$ for the polynomial in $\rH_{\et}^{2*}(X, \bQ_{\ell}(*))[x_1, \dots, x_i]$ defined by the same formula~\eqref{pBe} as over $\bC$, where 
$\rH_{\et}^{2*}(X, \bQ_{\ell}(*)) = \bigoplus_{i} \rH_{\et}^{2i}(X, \bQ_{\ell}(i))$. 
We say that a class in $\rH_{\et}^{2i}(X, \bQ_{\ell}(i))$ is \emph{integral} if it is in the image of $\rH_{\et}^{2i}(X, \bZ_{\ell}(i))$. 

\begin{theorem}
\label{theorem-lower-bound-general-k} 
Let $X$ be a smooth projective variety over the algebraic closure $k$ of a finitely generated field. 
Let $\ell \neq \characteristic(k)$ be a prime.
Let $\alpha \in \Br(X)$. 
Let 
$e$ be a positive integer and 
$\pi \colon P \to X$ a Severi--Brauer variety of class $\alpha$ and relative dimension at least $e$.  
\begin{enumerate} 
\item \label{divisible-1-k}
$\ind(\alpha)$ divides $e$ if and only if there exists a closed subvariety $Z \subset P$ whose generic fiber $Z_{\eta} \subset P_{\eta}$ is a twisted linear subvariety of codimension $e$. 
In particular, if $\ind(\alpha)$ divides $e$, there exists an integral Tate class $\gamma \in \rH_{\et}^{2e}(P, \bZ_{\ell}(e))$ whose restriction to fibers is the class of a codimension $e$ linear subspace. 
\item \label{divisible-2-k}
If $\alpha$ is $\ell$-adically topologically trivial and $B \in \rH^2_{\et}(X, \bQ_{\ell}(1))$ is a rational $\ell$-adic $B$-field for $\alpha$, then there exists a class $\gamma \in \rH^{2e}(P, \bZ_{\ell}(e))$ satisfying the above condition if and only if there exist Tate classes 
\begin{equation*} 
c_j \in \rH^{2j}(X, \bQ_{\ell}(j))^{\Tate}, 
\quad 1 \leq j \leq \min \set{ e, \dim(X) } , 
\end{equation*} 
such that $p_i^{B,e}(c_1, \dots, c_i) \in \rH^{2i}(X, \bQ_{\ell}(i))$ is integral for all $1 \leq i \leq \min \set{ e, \dim(X) }$. 
\end{enumerate} 
\end{theorem}

\begin{proof} 
The same argument as in our proof of Theorem~\ref{theorem-divisibility} in \S\ref{section-obstruction} works, using Lemma~\ref{lemma-cohomology-SB-k} and Remark~\ref{remark-coh-P} for~\eqref{divisible-2-k}. 
\end{proof} 

Using Theorem~\ref{theorem-lower-bound-general-k}, one can derive $\ell$-adic versions of the applications discussed in \S\ref{section-obstruction}--\S\ref{section-large-index}. 
Of these, we will only spell out the construction of Brauer classes with large index, parallel to Proposition~\ref{proposition-index-av}. 
For this, we will consider products of elliptic curves which are generic in the following sense, parallel to Lemma~\ref{lemma-generic-product}.

\begin{lemma}
\label{lemma-generic-product-Fp}
Let $p$ and $\ell$ be distinct primes. 
Let $X = E_1 \times \cdots \times E_g$ be a product of $g$ pairwise non-isogenous elliptic curves over $\overline{\bF}_p$. 
For each $1 \leq i \leq g$, let $x_i, y_i$ be a symplectic basis for $\rH_{\et}^1(E_i, \bZ_{\ell})$. Then: 
\begin{enumerate}
\item \label{T11-prod-Ei}
$\rH^{2}(X, \bZ_{\ell}(1))^{\Tate} = \bigoplus_{i=1}^g \bZ_{\ell}(x_i \cup y_i)$, where 
we denote the pullbacks of $x_i$ and $y_i$ to $X$ by the same letters.  
\item \label{Tate-prod-Ei}
The ring $\bigoplus_j \rH_{\et}^{2j}(X, \bQ_{\ell}(j))^{\Tate}$ is generated by $\rH_{\et}^{2}(X, \bQ_{\ell}(1))^{\Tate}$; 
explicitly,  
\begin{equation*}
\rH_{\et}^{2j}(X, \bQ_{\ell}(j))^{\Tate} = \bigwedge^j \rH_{\et}^{2}(X, \bQ_{\ell}(1))^{\Tate} = 
\bigoplus_{i_1 < \cdots < i_j} \bQ_{\ell}(x_{i_1} \cup y_{i_1} \cup \cdots \cup x_{i_j} \cup y_{i_j}). 
\end{equation*} 
\end{enumerate}
In particular, the $\ell$-adic integral Tate conjecture (see Conjecture~\ref{conjecture-ITC} below) holds in all degrees for~$X$. 
\end{lemma} 

\begin{proof}
\eqref{T11-prod-Ei} holds as the $E_i$ are pairwise non-isogenous, while~\eqref{Tate-prod-Ei} holds by~\cite[Example, page 158]{milne}.  
\end{proof} 

\begin{proposition}
\label{proposition-index-av-Fp}
Let $p$ and $\ell$ be distinct primes.  
Let $X = E_1 \times \cdots \times E_g$ be a product of $g \geq 2$ pairwise non-isogenous elliptic curves over $\overline{\bF}_p$. 
Let $1\leq t \leq g-1$ be an integer. 
Let $x_i, y_i$ be a symplectic basis for $\rH_{\et}^1(E_i, \bZ_{\ell})$, and let 
\begin{equation*}
    b = \sum_{i=1}^{t} x_i \cup y_{i+1} \in \rH_{\et}^2(X, \bZ_{\ell}(1)) . 
\end{equation*} 
For any $\ell$-th power $n$, let $\alpha = \varepsilon(b/n) \in \Br(X)[n]$ be the corresponding $\ell$-adic topologically trivial class. Then: 
\begin{enumerate}
\item \label{ind-divide-ng-1-Fp}
$\ind(\alpha)$ divides $n^{t}$. In fact, $\alpha$ is represented by an Azumaya algebra on $X$ of degree $n^{t}$. 
\item \label{ind-not-divide-ng-2-Fp}
If $n$ does not divide $(t-1)!$, then $\ind(\alpha) = n^{t}$. 
\end{enumerate}  
\end{proposition} 

\begin{proof}
The proof is identical to that of Proposition~\ref{proposition-index-av}, using Theorem~\ref{theorem-lower-bound-general-k} and Lemma~\ref{lemma-generic-product-Fp} in place of Theorem~\ref{theorem-lower-bound} and Lemma~\ref{lemma-generic-product}. 
\end{proof} 

\subsection{Counterexamples to the integral Tate conjecture}  
Finally, parallel to \S\ref{section-IHC}, we discuss the application of the above results to the integral Tate conjecture, and in particular the proof of Theorem~\ref{theorem-ITC} from the introduction. 
Let us recall the statement of the conjecture. 

\begin{conjecture}[$\ell$-adic integral Tate conjecture in codimension $c$] 
\label{conjecture-ITC}
Let $X$ be a smooth projective variety over the algebraic closure $k$ of a finitely generated field. 
Let $\ell \neq \characteristic(k)$ be a prime. 
Then the cycle class map $\CH^c(X) \otimes \bZ_{\ell} \to \rH^{2c}(X, \bZ_{\ell}(c))^{\Tate}$ is surjective. 
\end{conjecture} 

Like the integral Hodge conjecture, Conjecture~\ref{conjecture-ITC} is well-known to be false in general. 
To obtain the new counterexamples promised in Theorem~\ref{theorem-ITC}, 
we can mimic Theorem~\ref{theorem-IHC} over $\overline{\bF}_p$ using the examples from Proposition~\ref{proposition-index-av-Fp}. 
Let us explain this with attention to the details that are different over $\overline{\bF}_{p}$. 
First, there is an obvious $\ell$-adic analog of Proposition~\ref{proposition-0-cycle}, similar to how Theorem~\ref{theorem-lower-bound-general-k} is an analog of Theorem~\ref{theorem-divisibility}; for brevity we do not formulate this explicitly, 
but using it 
we obtain exactly as in Lemma~\ref{lemma-weaker-bound-AV} the following result. 

\begin{lemma}
\label{lemma-weaker-bound-AV-Fp}
Let $X$ be an abelian variety of dimension $g$ over the algebraic closure $k$ of a finitely generated field. 
Let $\ell \neq \characteristic(k)$ be a prime. 
Let $\alpha \in \Br(X)[n]$ be an $n$-torsion Brauer class where $n$ is an $\ell$-th power, 
and let $P \to X$ be a Severi--Brauer variety of class $\alpha$ and relative dimension $r$. 
If $n$ divides $(g-1)!$, then there exists an integral Tate class $\delta \in \rH_{\et}^{2r}(P, \bZ_{\ell}(r))^{\Tate}$ whose 
restriction to fibers of $P \to X$ has degree $n^{g-2}$. 
More precisely, 
if $b \in \rH^{2}_{\et}(X, \bZ_{\ell}(1))$ is an $\ell$-adic integral $B$-field of degree $n$ and $h$ is as in Lemma~\ref{lemma-cohomology-SB-k}\eqref{h-B-algebraic}, then 
\begin{equation}
\label{delta-av-p}
\delta = 
\begin{cases}
\sum_{i=0}^{g-1} \pi^* \left( n^{g-2-i}r(r-1)\cdots(r-(i-1)) \frac{b^i}{i!} \right) h^{r-i} & \text{if }  g \leq r \\
\sum_{i=0}^{r} \pi^* \left( n^{g-2-i}r(r-1)\cdots(r-(i-1)) \frac{b^i}{i!} \right) h^{r-i} & \text{if }  r < g. 
\end{cases} 
\end{equation} 
is such a class. 
\end{lemma} 

Note that for $\delta \in \rH_{\et}^{2r}(P, \bZ_{\ell}(r))^{\Tate}$ an arbitrary integral Tate class, the degree of the restriction of $\delta$ to fibers $P_{\bar{x}} \cong \bP_{\kappa(\bar{x})}^r$ is an element of $\bZ_{\ell}$ which in general may not lie in $\bZ$. 
We define a Tate-theoretic index, parallel to Definition~\ref{definition-ind-HP}, by considering the positive integer degrees that arise in this way. 

\begin{definition}
Let $X$ is a smooth projective variety over the algebraic closure $k$ of a finitely generated field. 
Let $\alpha \in \Br(X)$ and let $P \to X$ be a Severi--Brauer variety of class $\alpha$ and relative dimension $r$. 
For $\ell \neq \characteristic(k)$ a prime, the 
\emph{$\ell$-adic Tate-theoretic index of $\alpha$ with respect to $P$} is 
\begin{equation*}
\ind_{\ell{\text-}\Tate}^P(\alpha) \coloneqq \min  \set{ e \in \bZ_{> 0} \st \exists ~ \delta \in \rH_{\et}^{2r}(P, \bZ_{\ell}(r))^{\Tate} \text{ of degree $e$ on fibers of $P \to X$} }. 
\end{equation*} 
\end{definition} 

Then $\ind_{\ell{\text-}\Tate}^P(\alpha)$ divides $\ind(\alpha)$ by the $\ell$-adic analog of Proposition~\ref{proposition-0-cycle}. 
Moreover, when $\alpha$ is $\ell$-power torsion, 
$\ind_{\ell{\text-}\Tate}^P(\alpha)$ is closely related to the integral Tate conjecture. 
To see this, we will use the following 
$\ell$-adic variant of Lemma~\ref{lemma-index-SB}\eqref{index-0-cycle}. 

\begin{lemma}
\label{lemma-ind-0-cycle-ell}
Let $K$ be a field and $\alpha \in \Br(K)[\ell^{\infty}]$ an $\ell$-power torsion Brauer class for a prime $\ell$. 
Let $P$ be a Severi--Brauer variety of class $\alpha$. 
Then 
\begin{equation*}
\ind(\alpha) = \min \set{ e \in \bZ_{>0} \st \exists ~ Z \in \CH_0(P) \otimes \bZ_{\ell} \text{ such that } \deg(Z) = e } .  
\end{equation*} 
\end{lemma} 

\begin{proof}
By the characterization of $\ind(\alpha)$ in terms of $0$-cycles in Lemma~\ref{lemma-index-SB}\eqref{index-0-cycle}, it suffices to show that if $e \in \bZ_{>0}$ is realized as $\deg(Z) = e$ for
$Z \in \CH_0(P) \otimes \bZ_{\ell}$, then $\ind(\alpha)$ divides~$e$. 
Choose $0$-cycles $Z_i \in \CH_0(P)$ and $a_i \in \bZ_{\ell}$ such that $\sum a_i \deg(Z_i) = e$. 
We may assume that the $a_i$ are contained in $\bZ_{(\ell)}$. 
Indeed, if $I \subset \bZ_{(\ell)}$ is the ideal generated by the $\deg(Z_i)$, then 
its extension $I \bZ_{\ell} \subset \bZ_{\ell}$ contains $e$, but $I = (I \bZ_{\ell}) \cap \bZ_{(\ell)}$ since $\bZ_{(\ell)} \to \bZ_{\ell}$ is faithfully flat. 
Clearing denominators of the $a_i$, we find a cycle $W \in \CH_0(P)$ such that $\deg(W) = e d$ where $d$ is prime to $\ell$. 
Then by Lemma~\ref{lemma-index-SB}\eqref{index-0-cycle}, $\ind(\alpha)$ divides $ed$. 
However, $\ind(\alpha)$ is a power of $\ell$ by our assumption $\alpha \in \Br(K)[\ell^{\infty}]$, so we conclude $\ind(\alpha)$ divides $e$.  
\end{proof} 

\begin{lemma} 
Let $X$ be a smooth projective variety over the algebraic closure $k$ of a finitely generated field. 
Let $\ell \neq \characteristic(k)$ be a prime. 
Let $\alpha \in \Br(X)[\ell^{\infty}]$ and let $P \to X$ be a Severi--Brauer variety of class $\alpha$ and relative dimension $r$. 
\begin{enumerate}
\item \label{0-cycle-1-ell}
$\ind(\alpha)$ divides an integer $e$ if and only if there exists a cycle $Z \in \CH^r(P) \otimes \bZ_{\ell}$ whose generic fiber $Z_{\eta} \in \CH_0(P_{\eta}) \otimes \bZ_{\ell}$ is a $0$-cycle of degree $e$. 
\item \label{ind-ITC}
If $\ind_{\ell{\text-}\Tate}^P(\alpha) < \ind(\alpha)$, then the $\ell$-adic integral Tate conjecture in codimension $r$ fails on $P$. 
More precisely, if $\delta \in \rH_{\et}^{2r}(P, \bZ_{\ell}(r))^{\Tate}$ is a class whose degree on fibers is an integer which is not divisible by $\ind(\alpha)$, then $\delta$ is not algebraic, i.e. not in the image of the cycle class map $\CH^r(P) \otimes \bZ_{\ell} \to \rH_{\et}^{2r}(P, \bZ_{\ell}(r))^{\Tate}$. 
\end{enumerate} 
\end{lemma} 

\begin{proof}
\eqref{0-cycle-1-ell} follows from Lemma~\ref{lemma-ind-0-cycle-ell} by taking closures of $0$-cycles on the generic fiber, while~\eqref{ind-ITC} follows immediately from~\eqref{0-cycle-1-ell}. 
\end{proof} 

We obtain the following analogs of 
Corollary~\ref{corollary-ihc-fail-av} and 
Theorem~\ref{theorem-IHC} over $\overline{\bF}_p$.

\begin{corollary}
\label{corollary-itc-fail-av}
Let $X$ be an abelian variety of dimension $g$ over the algebraic closure $k$ of a finitely generated field. 
Let $\ell \neq \characteristic(k)$ be a prime. 
Let $\alpha \in \Br(X)[n]$ be an $n$-torsion Brauer class such that $n$ is an $\ell$-th power which divides $(g-1)!$ and 
$\ind(\alpha)$ does not divide $n^{g-2}$. 
Then for any Severi--Brauer variety $P \to X$ of relative dimension $r$, 
the $\ell$-adic integral Tate conjecture in codimension $r$ fails on $P$. 
More precisely, the class $\delta \in \rH_{\et}^{2r}(P, \bZ_{\ell}(r))^{\Tate}$ from~\eqref{delta-av-p} is not algebraic. 
\end{corollary} 

\begin{theorem}
\label{theorem-ITC-precise} 
Let $p$ be a prime and  
$X = E_1 \times \cdots \times E_g$ a product of $g \geq 3$ pairwise non-isogenous elliptic curves over $\overline{\bF}_p$.  
Let $\ell \neq p$ be a prime, let $n$ be an $\ell$-th power such that $n$ divides $(g-1)!$ but not $(g-2)!$, 
and let $\alpha \in \Br(X)[n]$ be the class from Proposition~\ref{proposition-index-av-Fp} for $t = g-1$. 
Let $P \to X$ be a Severi--Brauer variety of class $\alpha$ and relative dimension $n^{g-1}-1$, which 
exists by Proposition~\ref{proposition-index-av-Fp}\eqref{ind-divide-ng-1-Fp}. 
Then the $\ell$-adic integral Tate conjecture in codimension $n^{g-1}-1$ fails on $P$. 
More precisely, the class $\delta \in \rH_{\et}^{2r}(P, \bZ_{\ell}(r))^{\Tate}$ with $r = n^{g-1}-1$ from~\eqref{delta-av-p} is not algebraic. 
\end{theorem} 

\begin{proof}
By Proposition~\ref{proposition-index-av-Fp}\eqref{ind-not-divide-ng-2-Fp} $\ind(\alpha)$ does not divide $n^{g-2}$, so Corollary~\ref{corollary-itc-fail-av} applies. 
\end{proof} 

\begin{proof}[Proof of Theorem~\ref{theorem-ITC}]
Take $g = \ell + 1$ and $n = \ell$ in Theorem~\ref{theorem-ITC-precise}. 
Note that $\rH^*_{\et}(P, \bZ_{\ell})$ is torsion-free by Lemma~\ref{lemma-cohomology-SB-k}\eqref{integral-coh-P}. 
\end{proof} 


\providecommand{\bysame}{\leavevmode\hbox to3em{\hrulefill}\thinspace}
\providecommand{\MR}{\relax\ifhmode\unskip\space\fi MR }
\providecommand{\MRhref}[2]{%
  \href{http://www.ams.org/mathscinet-getitem?mr=#1}{#2}
}
\providecommand{\href}[2]{#2}


\end{document}